\documentclass{amsart}
\usepackage{amsmath, amssymb, enumerate} 

\usepackage{colonequals}

\usepackage{eucal}

\usepackage{tikz-cd}
\usetikzlibrary{decorations.pathmorphing}
\usepackage[nobysame]{amsrefs}

\usepackage{hyperref}
\hypersetup{%
  bookmarksnumbered=true,%
  colorlinks=true,%
  linkcolor=blue,%
  citecolor=blue,%
  filecolor=blue,%
  menucolor=blue,%
  urlcolor=blue,%
  bookmarksopen=true,%
  bookmarksdepth=2,%
  pageanchor=true}

\makeatletter
\@namedef{subjclassname@2020}{\textup{2020} Mathematics Subject Classification}
\makeatother

\numberwithin{equation}{section}

\theoremstyle{plain}
\newtheorem{theorem}[equation]{Theorem}

\newtheorem{proposition}[equation]{Proposition}
\newtheorem{lemma}[equation]{Lemma}
\newtheorem{corollary}[equation]{Corollary}

\theoremstyle{definition}

\newtheorem{example}[equation]{Example}

\newtheorem{chunk}[equation]{}

\theoremstyle{remark}

\newtheorem*{ack}{Acknowledgements}

\newtheorem*{claim}{Claim}

\newcommand{\bcone}{\mathbf{B}}
\newcommand{\ccone}{\mathbf{C}}

\newcommand{\CH}[3]{\operatorname{H}^{#1}(#2, #3)}
\newcommand{\codim}{\operatorname{codim}}
\newcommand{\coh}{\operatorname{coh}}
\newcommand{\cone}{\operatorname{cone}}
\newcommand{\dbcat}{\mathbf{D}^{\operatorname{b}}}
\newcommand{\dbbcat}{\mathbf{D}^{+}}
\newcommand{\depth}{\operatorname{depth}}
\newcommand{\grmod}{\operatorname{grmod}}
\newcommand{\kos}[2]{K(#1;#2)}
\newcommand{\HH}[2]{\operatorname{H}_{#1}(#2)}
\newcommand{\height}{\operatorname{height}}
\newcommand{\hilb}{\operatorname{Hilb}}
\newcommand{\idmap}{\operatorname{id}}
\newcommand{\length}{\ell}
\newcommand{\lch}[3]{\operatorname{H}_{#2}^{#1}(#3)}

\newcommand{\lotimes}{\otimes^{\mathbf{L}}}
\newcommand{\leftd}[1]{\mathbf{L}{#1}}
\newcommand{\pdim}{\operatorname{proj\,dim}}
\newcommand{\perf}{\mathcal{F}}
\newcommand{\proj}[1]{\operatorname{Proj}(#1)}
\newcommand{\rank}{\operatorname{rank}}
\newcommand{\reg}{\operatorname{reg}}
\newcommand{\shift}{\mathsf{\Sigma}}
\newcommand{\spec}{\operatorname{Spec}}
\newcommand{\supp}{\operatorname{Supp}}
\newcommand{\Tor}{\operatorname{Tor}}

\newcommand{\mcC}{{\mathcal C}}
\newcommand{\mcE}{{\mathcal E}}
\newcommand{\mcF}{{\mathcal F}}
\newcommand{\mcG}{{\mathcal G}}
\newcommand{\mcL}{{\mathcal L}}
\newcommand{\mcM}{{\mathcal M}}
\newcommand{\mcN}{{\mathcal N}}
\newcommand{\mcO}{{\mathcal O}}

\newcommand{\bfS}{{\mathbf S}}

\newcommand{\bbP}{{\mathbb P}}
\newcommand{\bbQ}{{\mathbb Q}}
\newcommand{\bbN}{{\mathbb N}}
\newcommand{\bbZ}{{\mathbb Z}}

\newcommand{\fm}{\mathfrak{m}}
\newcommand{\fp}{\mathfrak{p}}
\newcommand{\fq}{\mathfrak{q}}

\newcommand{\bsc}{\boldsymbol c}

\newcommand{\bst}{\boldsymbol t}
\newcommand{\bsx}{\boldsymbol x}
\newcommand{\bsu}{\boldsymbol u}

\newcommand{\vf}{\varphi}
\newcommand{\ve}{\varepsilon}

\newcommand{\lra}{\longrightarrow}
\newcommand{\thra}{\twoheadrightarrow}

\begin{document}

\title[Boij-S\"{o}derberg cones]{lim Ulrich sequences and  Boij-S\"{o}derberg cones}

\author[Iyengar]{Srikanth B. Iyengar}
\address{Department of Mathematics, University of Utah, Salt Lake City, UT 84112, USA}
\email{srikanth.b.iyengar@utah.edu}

\author[Ma]{Linquan Ma}
\address{Department of Mathematics, Purdue University, 150 N. University street, IN 47907}
\email{ma326@purdue.edu}

\author[Walker]{Mark E. Walker}
\address{Department of Mathematics, University of Nebraska, Lincoln, NE 68588, U.S.A.}
\email{mark.walker@unl.edu}

\begin{abstract}
This paper extends the results of Boij, Eisenbud, Erman, Schreyer, and S\"oderberg on the structure of Betti cones of finitely generated graded modules and finite free complexes over polynomial rings, to all finitely generated graded rings  admitting linear Noether normalizations. The key new input is the existence of lim Ulrich sequences of graded modules over such rings.
\end{abstract}

\date{\today}



\maketitle

\setcounter{tocdepth}{1}
\tableofcontents

\section{Introduction}
Let $k$ be a field and $A\colonequals k[x_1, \dots, x_d]$ a polynomial ring in $d$-variables, each of degree one. The \emph{Betti table} of a finitely generated graded $A$-module $M$ is the matrix of integers, denoted $\beta(M)$, that records the  homological positions  and the twists of the  graded free modules appearing in the minimal free resolution of $M$. Boij and S\"oderberg \cite{Boij/Soderberg:2008}
conjectured an explicit description  of the collection of all Betti tables of graded Cohen-Macaulay $A$-modules ``up to rational scaling''.
This conjecture was proved by Eisenbud and Schreyer \cite{Eisenbud/Schreyer:2009}, and subsequently extended to  all graded $A$-modules by Boij and S\"oderberg \cite{Boij/Soderberg:2012}.  In detail, for $0 \leq c \leq d$, let $\bcone^c(A)$ denote the cone
inside the $\bbQ$-vector space $\bigoplus_{0 \leq i \leq d, j \in \bbZ} \bbQ$ spanned by the Betti tables of all graded $A$-modules of codimension at least $c$. The results of Eisenbud and Schreyer, and Boij and S\"oderberg,  give an equality
\[
\bcone^c(A) = \bcone^c_d
\]
where $\bcone^c_d$ is the $\bbQ_{\geq 0}$-span of certain explicit Betti tables $\beta(\underline{a})$, where $\underline{a} = (a_0, \dots, a_l)$ ranges over all ``degree sequences'', that is to say,  $(l+1)$-tuples of increasing integers,  for $l \geq c$; see \eqref{eq:HK}
below for definition of $\beta(\underline{a})$. A Cohen-Macaulay module $M$ whose Betti table is a rational multiple of $\beta(\underline{a})$ for some choice of degree sequence $\underline{a}$ is called a ``pure module''. 

In this language, the theorem asserts that there exists a pure module for each possible degree sequence and $\bcone^c(A)$ is spanned by the Betti tables of pure modules.  Roughly speaking, this means that the Betti tables of all modules are governed by those of the  pure modules. In particular, the multiplicity conjecture of Huneke and Srinivasan is a direct consequence~\cite{Boij/Soderberg:2008}.

The main result of this paper extends these results to all $\bbN$-graded $k$-algebras that admit linear Noether normalizations, that is to say, a linear system of parameters; see \ref{ch:linear-sop}. We assume $k$ is infinite throughout this paper,  for the sake of simplicity, so this class of $k$-algebras includes the standard graded ones. For any graded $k$-algebra $R$ and integer $0 \leq c \leq d \colonequals \dim(R)$,  let $\bcone^c(R)$ be the cone spanned by Betti tables of modules of codimension at least $c$  and of finite projective dimension. We prove:

\begin{theorem}[Theorem~\ref{th:main-bs}]
\label{IntroThm1}
For $R$ and $\bcone^c_d$ as above, we have a containment
$\bcone^c(R) \subseteq \bcone^c_{d}$, and equality holds when $R$ is Cohen-Macaulay.
\end{theorem}

As in \cite{Boij/Soderberg:2008}, one can deduce the following bounds on the multiplicity  of a perfect module; that is to say, a finitely generated module whose projective dimension equals its codimension. When $R$ is Cohen-Macaulay, perfect modules are precisely those that are Cohen-Macaulay and of finite projective dimension. Thus the statement below generalizes the theorem on multiplicities over polynomial rings.

\begin{corollary}[Corollary~\ref{co:main-bs}]
With $R$ as above, any perfect $R$-module $M$ satisifes
\[
e(R) \frac{t_1 \cdots t_c}{c!} \leq e(M) \leq e(R) \frac{T_1 \cdots T_c}{c!}
\]
where $c$ is the codimension of $M$,  and the $t_i$'s and $T_i$'s are the minimal and maximal shifts occurring in the minimal free resolution of $M$. 
\end{corollary}

Eisenbud and Erman \cite{Eisenbud/Erman:2017} generalized the results in \cites{Boij/Soderberg:2008, Eisenbud/Schreyer:2009} to finite free complexes over $A$, and we extend their results to rings $R$ as above; see Theorem \ref{th:main-bs}.

The proof of Eisenbud and Schreyer uses a pairing between Betti tables of graded $A$-modules and cohomology tables of coherent sheaves on $\proj A = \bbP^{d-1}_k$, along with a description of the rational cone spanned by the cohomology tables of such coherent sheaves. Although we do not use such a pairing in this paper, we do establish a result concerning the cone of cohomology tables for coherent sheaves over $\proj R$ with $R$ as in Theorem \ref{IntroThm1}. When $k$ is an infinite perfect field of positive characteristic, we prove $\overline{\ccone}(\bbP^{d-1}_k) = \overline{\ccone}(\proj R)$, where $\overline{\ccone}$ is the component-wise limit closure of cohomology tables of coherent sheaves; see Theorem \ref{th:bs-cohomology}. 

The  results stated above are known when $R$ admits a graded Ulrich module, or, equivalently, when $\proj R$ admits an Ulrich sheaf; see \cite{Eisenbud/Erman:2017}. But the existence of such modules is known only in a small number of cases. Instead, we use  ``lim Ulrich sequences'' of graded modules, which are sequences of graded modules that asymptotically approximate Ulrich modules. They exist over any graded $k$-algebra admitting a linear Noether normalization, and where $k$ is a infinite perfect field of positive characteristic. The relevant results are established in sections \ref{se:lu-sheaves} and \ref{se:limU}.

\begin{ack}
 We thank Daniel Erman for suggesting  applications of lim Ulrich sequences to Boij-S\"oderberg theory, and for  discussions on this topic, and a referee for comments on an earlier version of this manuscript. The authors were partly supported by National Science Foundation grants DMS-2001368 (SB); DMS-1901672,  DMS-1952366, DMS-2302430 (LM); and DMS-1901848, DMS-2200732 (MW).  Ma was also partly supported by a fellowship from the Sloan Foundation.
\end{ack}

\section{Lim Ulrich sequences of sheaves}
\label{se:lu-sheaves}

For any noetherian scheme $X$ let  $\dbcat(X)$ be its derived category of coherent sheaves. The tensor product of coherent sheaves on $X$ is denoted  $-\otimes-$; its derived version is $-\lotimes-$. Throughout this manuscript $k$ is a field and $m\ge 1$ an integer. Let $\bbP^m$ be the projective space over $k$ of dimension $m$; we write $\bbP^m_k$ when the field $k$ needs emphasis.  

\begin{chunk}
\label{ch:pn}
For each coherent sheaf $\mcF$ (or a bounded complex of such) on $\bbP^m$ set
\[
\gamma_{i,t}(\mcF)\colonequals \rank_k \CH i{\bbP^m}{ \mcF(t)} \quad\text{for $i,t$ in $\bbZ$.}
\]
With $\ve\colon \bbP^m \to \spec(k)$ the structure map, for any coherent sheaf on $\mcF$  the counit of the adjoint pair $(\ve^*,\ve_*)$ is a natural map
\begin{equation}
\label{eq:counit}
\ve^*\ve_* \mcF\lra \mcF\,.
\end{equation}

The following result is well-known; see \cite[Proposition~2.1]{Eisenbud/Schreyer:2003}.

\begin{lemma}
\label{le:Ulrich}
Let $\mcF$ be a coherent sheaf  on $\bbP^m$. The conditions below are equivalent:
\begin{enumerate}[\quad\rm(1)]
\item
The map \eqref{eq:counit} is an isomorphism. 
\item
$\mcF\cong \mcO_{\bbP^m}^r$ for some integer $r\ge 0$.
\item
$\gamma_{i,t}(\mcF)=0$, except possibly when $i=0$ and $t\ge 0$, or  $i=m$ and $t\le -m-1$.
\end{enumerate}
When these conditions hold, $\mcF\cong \mcO_{\bbP^m}^{r}$ for $r=\gamma_{0,0}(\mcF)$.  \qed
\end{lemma}

Any nonzero sheaf $\mcF$ on $\bbP^m$ satisfying the equivalent condition of the result above is said to be an \emph{Ulrich sheaf}. Condition (3) is used to define Ulrich sheaves over arbitrary projective varieties.
\end{chunk}

\subsection*{Lim Ulrich sequences}
We introduce the notion of lim Ulrich sequences of sheaves on $\bbP^m$, following \cite[\S6.6]{Iyengar/Ma/Walker:2022a}.

\begin{chunk}
\label{ch:limUsheaves-P}
A \emph{lim Ulrich sequence} of sheaves on $\bbP^m$ is a sequence $(\mcF_n)_{n\geqslant 0}$ of coherent sheaves on $\bbP^m$ for which the following properties hold:
\begin{enumerate}[\quad\rm(1)]
\item
$\gamma_{0,0}(\mcF_n)\ne 0$ for each $n\ge 0$;
\item
There exists an integer $t_0$ such that $\gamma_{0,t}(\mcF_n)=0$ for $t\le t_0$ and all $n$;
\item
There exists an integer $t_1$ such that $\gamma_{\geqslant 1,t}(\mcF_n)=0$ for $t\ge t_1$ and all $n$;
\item
Except possibly when $i=0$ and $t\ge 0$, or  $i=m$ and $t\le -m-1$,  one has
\[
\lim_{n\to\infty}\frac{\gamma_{i,t}(\mcF_n)}{\gamma_{0,0}(\mcF_n)}=0\,.
\]
\end{enumerate}
The range of values of $i$ and $t$ arising in (4) is precisely the one from Lemma~\ref{le:Ulrich}. For geometric applications, the key properties are (1) and (4); see, in particular, Theorem~\ref{th:limU} below. The other conditions become important when considering the corresponding sequence of graded modules; see the proof of Theorem~\ref{th:limU-graded-modules}.

Clearly, if each $\mcF_n$ is an Ulrich sheaf, then the sequence $(\mcF_n)_{n\geqslant 0}$ is lim Ulrich. However, there are other lim Ulrich sequences on $\bbP^m_k$, at least when $k$ has positive characteristic, and these will be used to construct lim Ulrich sequences on a much larger family of schemes; see Theorem~\ref{th:limU-sheaves}. 
\end{chunk}

Next we present a characterization of lim Ulrich sequences on $\bbP^m$, in the spirit of the original definition of Ulrich sheaves.

\begin{chunk}
Let $\mcF$ be a coherent sheaf on $\bbP^m$. Consider $\cone(\mcF)$, the mapping cone  of the counit map \eqref{eq:counit}, and the induced exact triangle
\begin{equation}
\label{eq:cone}
\ve^*\ve_* \mcF\lra \mcF \lra \cone(\mcF) \lra 
\end{equation}
in $\dbcat(\coh \bbP^m)$. Observe that this is natural in $\mcF$, but it is not compatible with twists.  Keeping in mind that $\ve^*\ve_*\mcF\cong \mcO_{\bbP^m}^{r}$, tensoring the exact triangle defining $\cone(\mcF)$ with $\mcO_{\bbP}(t)$ yields an exact triangle
\[
\mcO_{\bbP^m}(t)^{r} \lra \mcF(t)\lra \cone(\mcF)(t) \lra
\]
The resulting exact sequence in cohomology reads
\begin{equation}
\label{eq:cone2}
\CH{*}{\bbP^m}{\mcO_{\bbP^m}(t)}^{r} \to \CH{*}{\bbP^m}{\mcF(t)} \to \CH{*}{\bbP^m}{\cone(\mcF)(t)} \to \CH{*+1}{\bbP^m}{\mcO_{\bbP^m}(t)}^{r}\,.
\end{equation}
\end{chunk}

The gist of the next result is that for a lim Ulrich sequence $(\mcF_n)_{n\geqslant 0}$, the counit maps $\ve^*\ve_* \mcF_n\to \mcF_n$ are asymptotically isomorphisms. If $\mcF$ is an Ulrich sheaf, applying it to the sequence $\mcF_n\colonequals \mcF$  and $\mcG=\mcO_{\bbP^m}$, recovers  Lemma~\ref{le:Ulrich}, but of course the latter  result, and its proof, are models for the one below. 

\begin{theorem}
\label{th:limU}
Let $(\mcF_n)_{n\geqslant 0}$ be a lim Ulrich sequence of sheaves on $\bbP^m$.
For any $\mcG$ in $\dbcat(\coh \bbP^m)$ and integers $i,j$, one has
\[
\lim_{n\to \infty}\frac{\gamma_{i,j}(\cone(\mcF_n)\lotimes\mcG)}{\gamma_{0,0}(\mcF_n)}=0\,.
\]
\end{theorem}

The proof only uses conditions (1) and (4) in \ref{ch:limUsheaves-P}. 

\begin{proof}
In what follows we write  $\bbP$ instead of $\bbP^m$, and $\mcO$ for the structure sheaf  on it. Set $r_n\colonequals \gamma_{0,0}(\mcF_n)$.  It helps to consider the collection $\bfS$ of sequences $(\mcC_n)_{n\geqslant 0}$ with $\mcC_n$ in $\dbcat(\coh \bbP)$, with the property that
\begin{equation}
\label{eq:limU}
\lim_{n\to \infty} \frac{\rank_k \CH{*}{\bbP}{\mcC_n}}{r_n}=0\,.
\end{equation}
The desired result is that the sequence with $\mcC_n \colonequals \cone(\mcF_n)\lotimes \mcG$ is in $\bfS$. In the proof we repeatedly use the following elementary observation: Given sequences  $(\mcC_n')_{n\geqslant 0}$, $(\mcC_n)_{n\geqslant 0}$, and $(\mcC_n'')_{n\geqslant 0}$, of bounded complexes of coherent sheaves such that for each $n$ there is an  exact triangle 
\[
\mcC_n'\lra \mcC_n \lra \mcC_n''\lra
\]
if two of the three sequences are in $\bfS$, then so is the third.

It follows from this observation that given an exact triangle $\mcG'\to \mcG\to \mcG''\to $ in $\dbcat(\coh \bbP)$, if the desired result holds for any two of $\mcG'$, $\mcG$, and  $\mcG''$, then it holds for the third. Since each object in the derived category is equivalent to a bounded complex consisting of twists of $\mcO$, a standard induction using these observations reduces checking the desired result to the case when $\mcG = \mcO(t)$, for $t\in\bbZ$; that is to say that the sequence $(\cone(\mcF_n)(t))_{n\geqslant 0}$ is in $\bfS$.
 
When $-m\le t\le -1$, one has $\CH {*}{\bbP}{\mcO(t)}=0$ so the exact sequence \eqref{eq:cone2} with $\mcF\colonequals \mcF_n$ gives the first equality below:
\[
\lim_{n\to \infty} \frac{\rank_k \CH{*}{\bbP}{\cone(\mcF_n)(t)}}{r_n}= \lim_{n\to \infty} \frac{\rank_k \CH{*}{\bbP}{\mcF_n(t)}}{r_n}=0\,.
\]
The second one is by condition \ref{ch:limUsheaves-P}(4). This is the desired result.

Suppose $t\ge 0$. We argue by induction on $t$ that the sequence $(\cone(\mcF_n)(t))_{n\geqslant 0}$ is in $\bfS$. The base case is $t=0$. Then since $\CH i{\bbP}{\mcO(t)}=0$ for all $i\ge 1$ the exact sequence \eqref{eq:cone2},  with $\mcF\colonequals \mcF_n$,  reduces to the exact sequence
\[
0\lra \CH{0}{\bbP}{\mcO(t)}^{r_n} \lra \CH{*}{\bbP}{\mcF_n(t)} \lra \CH{*}{\bbP}{\cone(\mcF_n)(t)} \lra 0
\]
When $t=0$ the map on the left is an isomorphism in degree $0$,  by the definition of $r_n$. Thus we get
\[
\rank_k \CH{*}{\bbP}{\cone(\mcF_n)} =  \sum_{i\ge 1}\rank_k \CH{i}{\bbP}{\mcF_n}
\]
and the desired limit is again immediate from \ref{ch:limUsheaves-P}(4).

Suppose $t\ge 1$ and consider the Koszul resolution of $\cone(\mcF_n)(t)$:
\[
0\to \cone(\mcF_n)(t-m-1) \to \cdots \to \cone(\mcF_n)^{m+1}(t-1) \to \cone(\mcF_n)(t)\to 0
\]
Since $t-m-1\ge -m$, from the already established part of the result and the induction hypotheses we get that the sequences $(\cone(\mcF_n)(j))_{n\geqslant 0}$ are in $\bfS$ for  $t-m-1\le j\le t-1$. The exact sequence above implies that the same holds for $j=t$.

This completes the discussion of the case $t\ge 0$.

Given this one can use the Koszul resolution above and a descending induction on $t$ to cover also the case $t\le -m-1$. 
\end{proof}

We record a construction of lim Ulrich sequences of sheaves on $\bbP^m$ from \cite[\S7]{Iyengar/Ma/Walker:2022a}.

\begin{chunk}
\label{ch:limU-p}
Let $k$ be an infinite perfect field of positive characteristic $p$. Set $Z\colonequals (\bbP^1_k)^m$. Let $\rho\colon Z\to \bbP^m_k$ be  the quotient by the action of the symmetric group $S_m$ on $Z$, or any  finite flat map. For each $n\ge 0$ set
\[
\mcL_n \colonequals \mcO_Z(p^n,2p^n,\dots,m p^n)\quad\text{and}\quad \mcE_n\colonequals \rho_*(\mcL_n)\,.
\]
The $\mcL_n$ are line bundles on $Z$, and the $\mcE_n$ are vector bundles on $\bbP_k^m$. The following result is a special case of \cite[Theorem~7.15]{Iyengar/Ma/Walker:2022a}. Here $\vf\colon \bbP^m_k\to \bbP^m_k$ is the Frobenius map; it is a finite map, since $k$ is perfect. 
 
\begin{theorem}
\label{th:limU-sheaves-P}
Let $k$ be an infinite perfect field of positive characteristic. Let $\mcN$ be a coherent sheaf on $\bbP^m_k$ of positive rank and  $\gamma_{0,t}(\mcN)=0$ for $t\ll 0$. With $\mcE_n$ the sheaf defined above, set
\[
\mcG_n \colonequals \vf^n_*(\mcN \otimes\mcE_n) \quad\text{for $n\ge 0$.}
\]
The sequence $(\mcG_n)_{\geqslant 0}$ of sheaves is lim Ulrich. \qed
\end{theorem}
\end{chunk}

\begin{chunk}
Fix a proper  $k$-scheme $X$  and $\mcL$ a  line bundle  on $X$. For any coherent sheaf $\mcF$ on $X$ and integer $t$ set
\[
\mcF(t) \colonequals \mcF \otimes \mcL^{\otimes t}\,.
\]
Recall that $\mcF$ is \emph{globally generated}  if there is a surjection $\mcO_X^r\to \mcF$ for some $r\ge 0$. The line bundle $\mcL$ is \emph{ample} if
for each coherent sheaf $\mcF$, the sheaf $\mcF(t)$ is globally generated for $t\gg 0$. The statement   below is well-known.

\begin{proposition}
\label{pr:ampleness}
Let $k$ be a field, $X$ a proper $k$-scheme, and $\mcL$ a line bundle on $X$. The following statements are equivalent.
\begin{enumerate}[\quad\rm(1)]
\item
$\mcL$ is ample and globally generated;
\item
There exists a finite map $\rho\colon X\to \bbP^n_k$ for some integer $n$ with  $\mcL \cong \rho^* \mcO_{\bbP^n}(1)$.
\end{enumerate}
If in addition $k$ is infinite, with $m\colonequals \dim X$,  these statements are  equivalent to
\begin{enumerate}[\quad\rm(3)] 
\item[\rm(3)] 
There exists a finite map $\pi\colon X \to \bbP^m_k$ and  $\mcL \cong \pi^* \mcO_{\bbP^m}(1)$.
\end{enumerate}
\end{proposition}

\begin{proof}
The equivalence of (1) and (2) is proved in  \cite[Proposition I.4.4 and (the proof of) Proposition 1.4.6]{Hartshorne:1970}.
 
 Suppose now that $k$ is infinite. The implication (3)$\Rightarrow$(2) is obvious.

(2)$\Rightarrow$(3): With $\rho$ as in (2), we can find a linear projection $\pi'\colon \bbP^n\to\bbP^m$ defined on all of $f(X)$, since $k$ is infinite. The map $\pi'\circ \rho$ has the desired properties.
\end{proof}
\end{chunk}

\begin{chunk}
\label{ch:good-pairs}
In what follows,  we say $(X,\mcL)$ is a \emph{Noether pair} to mean that $X$ is a proper scheme over an infinite field $k$, and $\mcL$ is an ample and globally generated line bundle on $X$.  The basic example is the pair  $(\bbP^m,\mcO(1))$ which we identify with $\bbP^m$. More generally, if $R$ is a graded $k$-algebra admitting a linear system of parameters, then the pair $(\proj R, \widetilde{R(1)})$ is a Noether pair; see Lemma~\ref{le:RtoX}.

A  morphism $f\colon (X,\mcL)\to (X',\mcL')$  of Noether pairs is a map of $k$-schemes $f\colon X\to X'$ such that  $f^*(\mcL')\cong \mcL$.    By Proposition~\ref{pr:ampleness}, given a Noether pair $(X,\mcL)$ there is a finite dominant linear map of Noether pairs
\[
\pi\colon (X,\mcL) \to \bbP^m_k \quad\text{where $m\colonequals \dim X$.}
\]
Thus, by the projection formula, for any sheaf $\mcF$ on $X$ one has
\begin{equation}
\label{eq:projection}
\CH iX{\mcF(t)} \cong \CH i{\bbP^m_k}{\pi_*\mcF(t)} \quad \text{for all $i,t$.}
\end{equation}
This allows one to introduce a notion of lim Ulrich sequence of sheaves on $X$; see \ref{ch:limUsheaves-X}. First we record the result  below, which is immediate from the equivalence (1)$\Leftrightarrow$(2) in Proposition~\ref{pr:ampleness}.

\begin{lemma}
\label{le:pairs-finitemaps}
Let $(X,\mcL)$ be a Noether pair and $f\colon X'\to X$ a finite morphism of $k$-schemes. Then  $(X',f^*{\mcL})$ is also a Noether pair. \qed
\end{lemma}
\end{chunk}

\begin{chunk}
\label{ch:limUsheaves-X}
Let  $(X,\mcL)$ be a Noether pair. An \emph{Ulrich} sheaf on $X$ is a coherent sheaf $\mcF$ such that the coherent sheaf $\pi_*(\mcF)$ on $ \bbP^m_k$ is Ulrich for some  finite dominant map $\pi\colon (X,\mcL)\to \bbP^m_k$ of Noether pairs; see \ref{ch:pn}. Given Lemma~\ref{le:Ulrich} and \eqref{eq:projection}, one can characterize this property purely in terms of the cohomology modules of twists of $\mcF$; in particular, it is independent of the choice of $\pi$.  In the same vein, a \emph{lim Ulrich sequence} of sheaves on $X$ is a sequence $(\mcF_n)_{n\geqslant 0}$ of coherent sheaves on $X$ for which the sequence $(\pi_*\mcF_n)_{\geqslant 0}$ is lim Ulrich in the sense of \ref{ch:limUsheaves-P}. Once again, this can be expressed purely in terms of the cohomology of the twists of $\mcF_n$, which reconciles the definition given here with that in \cite[6.6]{Iyengar/Ma/Walker:2022a}.

The result below serves to clarify that the lim Ulrich property is, to a certain extent, independent of the domain of definition of the sheaves involved.

\begin{lemma}
\label{le:limU-sheaves}
Let $f\colon (X,\mcL)\to (X',\mcL')$ be a morphism of Noether pairs where $f$ is finite and dominant. A sequence $(\mcF_n)_{n\geqslant 0}$ of coherent sheaves on $X$ is lim Ulrich if and only if the sequence $(f_*\mcF_n)_{n\geqslant 0}$ of coherent sheaves on $X'$ is lim Ulrich.
\end{lemma}

\begin{proof}
The stated claim is immediate from the observation that if $\pi'\colon X'\to \bbP^m$ is a finite dominant linear map, then so is the composition $\pi'\circ f$.
\end{proof}

\end{chunk}
The result below is a minor extension of \cite[Theorem~7.15]{Iyengar/Ma/Walker:2022a}.

\begin{theorem}
\label{th:limU-sheaves}
Let $(X,\mcL)$ be a Noether pair with $\dim X\ge 1$. Assume furthermore that the field $k$ is infinite and perfect of positive characteristic. There exists a lim Ulrich sequence of sheaves  $(\mcF_n)_{n\geqslant 0}$ on $X$ such that for each $n\ge 0$ one has
\[
\depth (\mcF_n)_x \ge
\begin{cases}
1 & \text{when $x\in X$ is a closed point}\\
\depth \mcO_{X,x} & \text{for all $x$}.
\end{cases}
\]
\end{theorem}

\begin{proof}
Set $m\colonequals \dim X$ and let $\pi\colon  X\to \bbP^m_k$ be a finite dominant linear map.  We begin by choosing a coherent sheaf $\mcM$  on $X$ satisfying the following conditions
\begin{enumerate}[\quad\rm(1)]
\item
$\CH 0{\bbP^m_k}{\pi_*\mcM(t)} = \CH 0X{\mcM(t)}=0$ for $t\ll 0$;
\item
The sheaf $\pi_*\mcM$ on $\bbP^m_k$ has positive rank;
\item
$\depth \mcM_x \ge \depth \mcO_{X,x}$ for each $x\in X$.
\end{enumerate}

To construct such an $\mcM$, consider the closed subset $W$ of $X$  consisting of  closed points $x\in X$ such that $\depth \mcO_{X,x}=0$.  Let $\mathcal{J}$ be the sheaf of local sections of $\mcO_{X}$ supported on $W$, and set $\mcM\colonequals \mcO_X/\mathcal{J}$.  This construction ensures  $\mcM_x$ is isomorphic to $\mcO_{X,x}$ for $x\not\in W$, and $\depth \mcM_x \ge 1$ for $x\in W$.  This justifies both (1) and (3); see Lemma~\ref{le:sheaf-depth} below.  Moreover $\dim \mcM = \dim X = \dim \bbP^m_k$, so $\pi_*\mcM$ has positive rank and (2) holds. 

With $\mcE_n$ the sheaves on $\bbP^m_k$ defined in \ref{ch:limU-p}, set
\[
\mcF_n \colonequals \vf^n_*(\mcM \otimes\pi^*\mcE_n)
\]
where $\vf$ is the Frobenius morphism on $X$. Since the Frobenius map preserves depth and $\pi^*\mcE_n$ is locally free,  $\depth (\mcF_n)_x=\depth \mcM_x$ for each $n$ and $x\in X$. In particular $\depth(\mcF_n)_x$ has the stated properties.

Using the fact that Frobenius commutes with pushfoward of sheaves, and the projection formula, one gets
\[
\pi_* \mcF_n = \pi_*\vf^n_*(\mcM \otimes\pi^*\mcE_n) \cong \vf^n_*(\pi_*\mcM \otimes \mcE_n)\,.
\]
Theorem~\ref{th:limU-sheaves-P} yields that the sequence $(\pi_*\mcF_n)_{n\geqslant 0}$ of sheaves on $\bbP^m_k$ is lim Ulrich. Thus, by definition, the sequence $(\mcF_n)_{n\geqslant 0}$ is lim Ulrich; see also Lemma~\ref{le:limU-sheaves}.
\end{proof}

\begin{lemma}
\label{le:sheaf-depth}
Let $X$ be as in \eqref{ch:good-pairs} and $\mcF$ a coherent  sheaf on $X$. If $\mcF_x$ has positive depth  at every closed point $x\in X$, then $\CH 0X{\mcF(t)} = 0$ for $t \ll 0$.
\end{lemma}

\begin{proof}
 By pushing forward to $\bbP^m$ we may reduce to the case where $X=\bbP^m$. In this case, Grothendieck-Serre duality gives
\[
\CH 0X{\mcF(t)}^* \cong \mathrm{Ext}^m(\mcF, \mcO_{X}(-t-m-1))
\]
where $(-)^*$ denotes $k$-linear duals and Ext is computed in  the abelian category of coherent sheaves.  We have a spectral sequence
\[
\mathrm{E}^{p,q} \colonequals \CH p{X}{\underline{\mathrm{Ext}}^q(\mcF, \mcO_{X}(-t-m-1))}
\Longrightarrow \mathrm{Ext}^{p+q}(\mcF, \mcO_{X}(-t-m-1))  
\]
where $\underline{\mathrm{Ext}}$ denotes sheafified Ext. Since $\mcF_x$ has  positive depth at all closed points and the local rings $\mcO_{X,x}$ are regular, we have  $\underline{\mathrm{Ext}}^m(\mcF, \mcO_{X}) = 0$; thus $\mathrm{E}^{0,m} = 0$.  By Serre vanishing  $\mathrm{E}^{p,m-p} = 0$ for all $p \ge 1$ for $t \ll 0$. This proves 
\[
\mathrm{Ext}^m(\mcF, \mcO_{X}(-t-m-1))\quad\text{for $t\ll 0$.} \qedhere
\]
\end{proof}

\section{Cones of cohomology tables}
\label{se:cohomology}
Throughout this section $k$ is an infinite field and  $(X,\mcL)$ a Noether pair over $k$, as in~\ref{ch:good-pairs}, and  $\pi\colon  X\to \bbP^m_k$ a finite dominant  map with $\mcL \cong \pi^* \mcO_{\bbP_k^m}(1)$; thus $m=\dim X$. We prove that the closure of the cohomology tables of coherent sheaves on $X$ and $\bbP_k^m$ coincide; see~Theorem~\ref{th:bs-cohomology}.

\begin{chunk}
\label{ch:notation-space}
Extending the notation from \ref{ch:pn},  for any $\mcG$ in $\dbcat(\coh X)$ set
\[
\gamma_{i,j}(\mcG) \colonequals \rank_k \CH{i}X{\mcG(j)}\,\quad \text{and}\quad \gamma(\mcG)\colonequals (\gamma_{i,j})_{i,j}\,.
\]
By the projection formula \eqref{eq:projection}, which applies also to objects in $\dbcat(\coh X)$, any $\gamma(\mcG)$ occurs as the cohomology table of a complex in $\dbcat(\coh \bbP^m)$. The result below means that the converse also holds, up to limits.

\begin{lemma}
\label{le:bs-cohomology}
Let  $(\mcF_n)_{n\geqslant 0}$ be a  lim Ulrich sequence of sheaves on $X$.  For each $\mcG$ in $\dbcat(\coh \bbP^m_k)$ there exists a sequence $(\mcG_n)_{n\geqslant 0}$ in $\dbcat(\coh X)$ such that 
\[
\lim_{n\to \infty} \frac{\gamma_{i,j}(\mcG_n)}{\gamma_{0,0}(\mcF_n)} = \gamma_{i,j}(\mcG)\qquad\text{for all $i,j$.}
\]
When $\mcG$ is a vector bundle on $\bbP^m$, the $\mcG_n$ can be chosen in $\coh X$.
\end{lemma}

\begin{proof}
For each $n\ge 0$ set
\[
\mcG_n\colonequals  \mcF_n  \lotimes \leftd\pi^*\mcG\,.
\]
In particular, when $\mcG$ is a vector bundle, each $\mcG_n$ is a coherent sheaf on $X$. The projection formula yields
\[
\pi_*(\mcG_n) \cong  \pi_*(\mcF_n)  \lotimes \mcG  \,,
\]
so  Theorem~\ref{th:limU} applied to the lim Ulrich sequence $(\pi_*(\mcF_n))_{n\geqslant 0}$ yields 
\[
\lim_{n\to \infty}\frac{\gamma_{i,j}(\mcG_n)}{\gamma_{0,0}(\mcF_n)} 
	= \lim_{n\to \infty}\frac{\gamma_{i,j}(\ve^*\ve_*\pi_*(\mcF_n)  \lotimes \mcG)}{\gamma_{0,0}(\mcF_n)}  \,.
\]
Since one has equalities
\[
\gamma_{i,j}( \ve^*\ve_*\pi_*(\mcF_n)  \lotimes \mcG ) = \gamma_{0,0}(\mcF_n) \gamma_{i,j}(\mcG)
\]
the desired result follows.
\end{proof}

\end{chunk}

\begin{chunk}
\label{ch:cone-cohomology}
Fix an integer $m$ and consider the $\bbQ$-vector space
\[
W \colonequals \bigoplus_{i=0}^m \prod_{j\in\bbZ}\bbQ\,,
\]
endowed with the topology defined by pointwise convergence.

For any coherent sheaf $\mcG$ on $X$ we view $\gamma(\mcG)$, defined in \ref{ch:notation-space}, as an element in $W$. This is the \emph{cohomology table} of $\mcG$.  The \emph{cone} of cohomology tables is the subspace
\[
\ccone(X)\colonequals \sum_{\mcG\in \coh X} \bbQ_{\geqslant 0} \gamma(\mcG)\,,
\]
of $W$. Observe that the definition involves only the coherent sheaves on $X$, and not all of $\dbcat(\coh X)$. We write $\overline{\ccone}(X)$ for the closure of $\ccone(X)$ in $W$.
\end{chunk}

\begin{theorem}
\label{th:bs-cohomology}
Let $k$ be an infinite perfect field of positive characteristic and $(X,\mcL)$ a Noether pair over $k$. There is an equality $\overline{\ccone}(X)=\overline{\ccone}(\bbP^m)$ for $m\colonequals \dim X$.
\end{theorem}

\begin{proof}
Let $\pi\colon (X,\mcL) \to \bbP^m_k$ be a finite dominant linear map of Noether pairs. It is clear from \eqref{eq:projection} that $\ccone(X)\subseteq \ccone(\bbP^m)$. We have only to prove that the cohomology table of a coherent sheaf $\mcF$ on $\bbP^m$ is in  $\overline{\ccone}(X)$. By \cite[Theorem~0.1]{Eisenbud/Schreyer:2010}, one has a convergent sum
\[
\gamma(\mcF) = \sum_{i\geqslant 1}a_i\gamma(\mcE_i) 
\]
where each $\mcE_i$ is a coherent sheaf obtained as a push-forward of a vector bundle on a linear subset of $\bbP^m$ and each $a_i$ is a positive real number. We claim that it suffices to verify that each $\gamma(\mcE_i)$ is in $\overline{\ccone}(X)$.

Indeed, assume this is so and for each $n\ge 1$ consider the sum
\[
\gamma(n) \colonequals a_1\gamma(\mcE_i) + \cdots + a_n\gamma(\mcE_n)\,.
\]
Since each $\gamma(\mcE_i)$ is in $\overline{\ccone}(X)$ approximating the $a_i$ by positive rational numbers we can write 
\[
\gamma(n) = \lim_{s\to \infty} \gamma(n)(s)
\]
where each $\gamma(n)(s)$ is a positive rational combination of $\gamma(\mcE_1),\dots,\gamma(\mcE_n)$ and hence in $\overline{\ccone}(X)$. For each $n$ pick an integer $s_n\ge 1$ such that
\[
|\gamma(n)_{a,b} - \gamma(n)(s_n)_{a,b}| < 1/2^n
\]
for all $a,b$ with $|a|,|b|< n$. Evidently  
\[
\lim_{n\to \infty} \gamma(n)(s_n) = \gamma(\mcF)
\]
and each $\gamma(n)(s_n)$ is in $\overline{\ccone}(X)$, so $\gamma(\mcF)$ is in $\overline{\ccone}(X)$, as claimed.

It thus remains to consider a coherent sheaf of the form $\iota_*(\mcE)$, where $\iota\colon \bbP^c\subseteq \bbP^m$ is a linear subspace and $\mcE$ is a vector bundle on $\bbP^c$, and verify that   $\gamma(\mcE)$ is in the closure of $\ccone(X)$.

Consider the scheme $X'$ obtained by the pull-back of $\pi$ along $\iota$. 
\[
\begin{tikzcd}
X' \ar[d,"\pi'" swap] \ar[r] & X  \ar[d,"\pi"] \\
\bbP^c  \ar[r, "\iota"] & \bbP^m
\end{tikzcd}
\]
With $\mcL'\colonequals (\pi')^*\mcO_{\bbP^c}(1)$, the pair $(X',\mcL')$ is a Noether pair; see Lemma~\ref{le:pairs-finitemaps}. Using the projection formula we identify  the cohomology tables of coherent sheaves on $X'$ and on $\bbP^c$ as elements of $W$,  from \ref{ch:cone-cohomology}. Thus $\overline{\ccone}(X')\subseteq \overline{\ccone}(X)$.

The scheme $X'$ admits a lim Ulrich sequence. This is clear when $c=0$; for $c\ge 1$ it is contained in Theorem~\ref{th:limU-sheaves}.  Thus, keeping in mind that $\mcE$ is a vector bundle, Lemma~\ref{le:bs-cohomology} yields that $\gamma(\mcE)$ is in $\overline{\ccone}(X')$, and hence in $\overline{\ccone}(X)$, as desired.
\end{proof}

\section{Lim Ulrich sequences of graded modules}
\label{se:limU}
As before $k$ is a field. Let $R\colonequals \{R_i\}_{i\geqslant 0}$ be a finitely generated, graded $k$-algebra with $R_0=k$. Set $\fm\colonequals R_{\geqslant 1}$; this is the  unique homogenous  maximal ideal of $R$.  We write $\grmod R$ for the category of finitely generated graded $R$-modules, with morphisms the degree preserving $R$-linear maps. The component in degree $j$ of a graded $R$-module  $M$ is denoted $M_j$. 

The full subcategory of the derived category of $\grmod R$ consisting of complexes $C$ with $\HH iC=0$  for $i\ll 0$ is denoted $\dbbcat(\grmod R)$. It contains   $\dbcat(\grmod R)$ the bounded derived category of finitely generated graded $R$-modules.

\begin{chunk}
\label{ch:uequivalence}
 Let $\bsu \colonequals  (u_n)_{n\geqslant 1}$ be a sequence of positive integers. We say that a sequence of objects $(C^n)_{n\geqslant 1}$ of $\dbbcat(\grmod R)$ is \emph{$\bsu$-trivial} if for all integers $i,j$ one has
\[
\lim_{n \to \infty} \frac{\rank_k \HH i{C^n}_j}{u_n} = 0\,.
\]
A sequence of morphisms $(f^n\colon C^n \to D^n)$ in $\dbbcat(\grmod R)$ is a \emph{$\bsu$-equivalence} if the sequence of their mapping cones,  $\cone(f^n)$, is $\bsu$-trivial. It is easy to verify that a composition of $\bsu$-equivalences is also a $\bsu$-equivalence. This observation will be used often in the sequel.
\end{chunk}

The proof of the result below is straightforward; see also \cite[Lemma~5.3]{Iyengar/Ma/Walker:2022a}.

\begin{proposition}
\label{pr:uiso}
 If a sequence of morphisms $(f^n\colon C^n \to D^n)$ in $\dbbcat(\grmod R)$ is a $\bsu$-equivalence, then for any $i,j$ there is an equality
 \[
\limsup_{n \to \infty} \frac{\rank_k \HH i{C^n}_j}{u_n}  = \limsup_{n \to \infty} \frac{\rank_k \HH i{D^n}_j}{u_n}\,.
\]
The corresponding statements involving $\liminf$ also holds. In particular, if one of the limits exists, so does the other and the two limits coincide. \qed
\end{proposition}

We say that a complex of graded $R$-modules is \emph{perfect} if it is quasi-isomorphic to a bounded complex of finitely generated free $R$-modules.

\begin{proposition}
\label{pr:uideal}
Let $(f^n\colon C^n\to D^n)$ be a $\bsu$-equivalence. For  $P$ in $\dbcat(\grmod R)$,  the induced sequence
\[
(\idmap \otimes f_n\colon  P\lotimes_R C^n \to P \lotimes_R D^n)
\]
is also an $\bsu$-equivalence under either of the following conditions: 
\begin{enumerate}[\quad\rm(1)]
\item
The complex $P$ is perfect;
\item
There exists an integer $s$ such that $\HH i{\cone(f_n)}=0$ for all $i<s$ and all $n$.
\end{enumerate}
\end{proposition}

\begin{proof}
Replacing $P$ by its minimal free resolution, we get that $P_i$, the component in homological degree $i$,  is finite free with $P_i=0$ for $i\ll 0$.  The desired conclusion is that the sequence $(P\otimes_R \cone(f_n))$ is $\bsu$-trivial under either of the two conditions.  

When $P$ is perfect, $P_i=0$ also for $i\gg 0$, and  the $\bsu$-triviality can be checked by a straightforward induction on the  number of nonzero components of $P$.

Suppose condition (2) holds, and let $s$ be as given; we may assume $s=0$. Fix an integer $i$. It is easy to see that for each integer $n$, the inclusion $P_{\leqslant i+1}\subseteq P$ of sub-complexes induces an isomorphism
\[
\HH i{P_{\leqslant i+1}\otimes_R\cone(f_n)}  \xrightarrow{\ \cong \ } \HH i{P\otimes_R \cone(f_n)}\,.
\]
Since $P_{\leqslant i+1}$ is perfect, one can invoke the already established part (1).
\end{proof}

In the sequel we need the theory of multiplicities for graded modules. Since the rings we work with are not necessarily standard-graded, we begin by recalling the basic definitions and results in the form we need.

\begin{chunk}
\label{ch:linear-sop}
Let $R$ be a graded $k$-algebra as above. For any finitely generated graded $R$-module $M$ and integer $q\ge \dim_RM$ set
\[
e_q(M)\colonequals q! \lim_{n\to\infty} \frac{\length_R(M/\fm^{n+1} M)}{n^{q}}\,.
\]
This is an integer, and it is equal to $0$ when $q\ge \dim_RM+1$. The \emph{multiplicity} of $M$ is the integer
\[
e(M)\colonequals e_q(M) \quad\text{for $q=\dim_RM$.}
\]
This coincides with the multiplicity of  $M_{\fm}$ as a module over the local ring $R_{\fm}$.

In what follows the focus is on graded $k$-algebras $R$ admitting a \emph{linear} system of parameters; that is to say, a system of parameters in $R_1$. It is convenient to call the $k$-subalgebra generated by a linear system of parameters a \emph{linear Noether normalization} of $R$.  Linear Noether normalizations exist, for instance, when $R$ is standard-graded (homogeneous, in the language of \cite[\S4.1]{Bruns/Herzog:1998}) and $k$ is infinite. The next results allow one to reduce computing multiplicities over rings admitting linear Noether normalizations to the case of standard graded rings. 

\begin{lemma}
\label{le:multiplicity0}
Let $R$ be a finitely generated graded $k$-algebra, and set $\fm\colonequals R_{\geqslant 1}$. For any set of elements $\bsx\subseteq R_1$ satisfying $\sqrt{(\bsx)}=\fm$, the ideal $(\bsx)$ is a reduction of $\fm$; it is a minimal reduction when $\bsx$ is a system of parameters for $R$.
\end{lemma}

\begin{proof}
The claim about minimality follows from \cite[Corollary~8.3.6]{Huneke/Swanson:2006} once we prove that $(\bsx)$ is a reduction of $\fm$, for the analytic spread of $\fm$ equals $\dim R$.

The desired result is clear when $R$ is standard graded because $\fm^n = (\bsx) \fm^{n-1}$ for any $n$ such that $\fm^n \subseteq (\bsx)$, and such an $n$ exists since $\fm$ is the radical of $(\bsx)$.

In what follows given a homogenous ideal $I$ in $R$ such that $\length_R(M/IM)$ is finite, we write $e(I,M)$ for the multiplicity of $M$ with respect to $I$. Namely, the limit of the sequence in \ref{ch:linear-sop} with $\fm$ replaced by $I$; see also \cite[Definition~4.6.1]{Bruns/Herzog:1998}. 

The goal is to prove that $\fm$ is the integral closure of $(\bsx)$; see \cite[Corollary~1.2.5]{Huneke/Swanson:2006}. It suffices to verify that this property holds modulo each minimal prime of $R$; see \cite[Proposition~1.1.5]{Huneke/Swanson:2006}. Thus we may assume $R$ is a domain, and then the desired result is equivalent to: $e(\bsx,R)=e(R)$; this is by Rees' theorem~\cite[Theorem~11.3.1]{Huneke/Swanson:2006}.

There exists an integer $s$ such that the Veronese subring $R^{(s)}$ of $R$ is standard-graded, after rescaling. Thus $\bsx^{s}\colonequals x_1^s,\dots,x_d^s$ is a reduction of the maximal ideal $(R_s)R^{(s)}$ of $R^{(s)}$, and hence one gets that $e(\bsx^s ,R^{(s)})=e(R^{(s)})$. This justifies the second equality below:
\[
e(\bsx^s, R) 
	= s\cdot e(\bsx^s, R^{(s)}) 
	= s\cdot e((R_s), R^{(s)})
	= e((R_s), R)
\]
The first and the third equalities hold as the rank of $R$ over $R^{(s)}$ is $s$; see \cite[Corollary~4.7.9]{Bruns/Herzog:1998}. 
If $f_1,\dots, f_n$ are homogeneous generators of the ideal $\fm$, then  $\fm^s$ is integral over $(f_1^s,\dots,f_n^s)$, which is contained in the ideal $(R_s)$, so we get that
\[
e((R_s), R) \le e((f_1^s,\dots,f_n^s),R) = e(\fm^s,R)\,.
\]
 We conclude that $e(\bsx^s,R)\leq e(\fm^s,R)$ and hence that $e(\bsx, R)\le e(\fm, R)$. The reverse inequality is clear. 
\end{proof}

\begin{lemma}
\label{le:multiplicity}
Let $R$ be a graded $k$-algebra admitting a linear system of parameters $\bsx$. For each $M$  in $\grmod R$ there is an equality
\[
e(M) = (\dim_RM)! \lim_{n\to\infty} \frac{\sum_{i\leqslant n}\rank_k(M_i)}{n^{\dim_RM}}\,.
\]
When $\dim_RM=\dim R$, one has also an equality
\[
e(M) = \sum_{i\geqslant 0} (-1)^i \rank_k \HH i{\bsx; M}\,.
\]
\end{lemma}

\begin{proof}
By Lemma~\ref{le:multiplicity0} the ideal $(\bsx)$ is a minimal reduction of the maximal ideal of $R$, so the multiplicity of $M$ as an $R$-module coincides with its multiplicity when viewed as a module over the $k$-subalgebra of $R$ generated by $\bsx$. Therefore, in the remainder of the proof, we assume that $R$ is a standard graded polynomial ring, over indeterminates $\bsx$. Then the fact the multiplicity is the Euler characteristic of the Koszul homology of $M$ with respect to $\bsx$ is \cite[Theorem~4.7.6]{Bruns/Herzog:1998}.

The first equality is also well-known, but we could not find a suitable reference, so we outline a proof: Set $q\colonequals \dim_RM$ and let 
$e'_q(-)$ be the function defined on the subcategory of finitely generated $R$-modules of dimension at most $q$ by
\[
e'_q(N) \colonequals q!  \lim_{n\to\infty} \frac{\sum_{i\leqslant n}\rank_k(N_i)}{n^q}\,.
\]
The desired result is that $e'_q(-)=e_q(-)$ on this subcategory of $R$-modules. It is clear that equality holds when the module is generated in a single degree. Since any finitely generated module admits a finite filtration by $R$-submodules such that the subquotients in the filtrations are generated in a single degree, it remains to note that both invariants are additive on short exact sequences of modules of dimension at most $q$; this is clear for $e'_q(-)$ and for $e_q(-)$ it is \cite[Corollary~4.7.7]{Bruns/Herzog:1998}. 
\end{proof}

\end{chunk}

\begin{chunk}
\label{ch:Ulrich}
Let $R$ be a graded $k$-algebra admitting a linear Noether normalization and $U$ a finitely generated graded $R$-module. We write $\nu_R(U)$ for its minimal number of generators. The $R$-module $U$ is \emph{Ulrich} if it has the following properties:
\begin{enumerate}[\quad\rm(1)]
\item
$U$ is maximal Cohen-Macaulay and nonzero;
\item
$e(U)=\nu_R(U)$;
\item
$U$ is generated in degree $0$.
\end{enumerate}
When $R$ is a  standard graded polynomial ring, the only Ulrich modules are $R^n$, for some integer $n$; no twists are allowed.
\end{chunk}

Following \cites{Iyengar/Ma/Walker:2022a, Ma:2023}, we introduce a notion lim Ulrich sequences of graded $R$-modules. To that end, we recall some basic properties of finite free complexes.

\begin{chunk}
\label{ch:short-complexes}
Consider a complex 
\[
F\colonequals 0\lra F_n \lra F_{n-1}\lra \cdots \lra F_1 \lra F_0 \lra 0
\]
of finite free $R$-modules  with $\rank_k \HH {}F$ finite and nonzero. The New Intersection Theorem~\cite{Roberts:1987} yields $n\ge \dim R$; if equality holds we say $F$ is a \emph{short} complex. For such an $F$, any maximal Cohen-Macaulay $R$-module $M$ satisfies 
\[
\HH i{F\otimes_RM}=0 \qquad\text{for $i\ge 1$.}
\]
This is by the acyclicity lemma~\cite[Exercise~1.4.24]{Bruns/Herzog:1998}. This property characterizes maximal Cohen-Macaulay modules, by the depth sensitivity of Koszul complexes.
\end{chunk}

\subsection*{Lim Ulrich sequences}
Let $R$ be a graded $k$-algebra admitting a linear system of parameters. Compare  the definition below with that of an Ulrich module ~\ref{ch:Ulrich}, and also with that of lim Ulrich sheaves~\ref{ch:limUsheaves-P}.

\begin{chunk}
\label{ch:lim-Ulrich}
A sequence  $(U^n)_{n\geqslant 1}$ in $\grmod R$  is \emph{lim Ulrich} if each $U^n$ is nonzero and the following properties hold:
\begin{enumerate}[\quad\rm(1)]
\item
For each finite free short complex  $F$ with $\rank_k \HH{}F$ finite one has
\[
\lim_{n\to\infty} \frac{\rank_k \HH i{F\otimes_RU^n}}{\nu_R(U^n)}=0 \qquad\text{for  $i\ge 1$.}
\]
\item
$\lim_{n\to\infty} e_d(U^n)/\nu_R(U^n)=1$ for $d\colonequals \dim R$.
\item
With $\fm\colonequals R_{\geqslant 1}$ the unique homogenous maximal ideal of $R$, one has
\[
\lim_{n\to\infty} \frac{\rank_k(U^n/\fm U^n)_0}{\nu_R(U^n)}=1\,.
\]
\end{enumerate}
 It suffices that condition (1) hold for $F$ the Koszul complex on a \emph{single} system of parameters for $R$; this can be proved along the lines of \cite[Lemma~5.7]{Iyengar/Ma/Walker:2022a}.  Moreover, with $F\colonequals K(\bsx)$  the Koszul complex on a linear system of parameters $\bsx$, parts (1) and (2) above can be expressed more succinctly as follows:
\[
\label{eq:lim-Ulrich}
\lim_{n\to\infty} \frac{\rank_k \HH i{\bsx; U^n}}{\nu_R(U^n)}=
\begin{cases}
1 & \text{when $i=0$}\\
0  & \text{when  $i\ne 0$.}
\end{cases} \tag{1$'$}
\]
Here is an expression of these conditions in the language introduced in \ref{ch:uequivalence}.

\begin{lemma}
\label{le:limU-check}
Let $R$ be a graded $k$-algebra with a linear system of parameters  $\bsx$, and $(U^n)_{n\geqslant 1}$ a lim Ulrich sequence in $\grmod R$. For $\bsu\colonequals (\nu_R(U^n))$, the following sequences of canonical surjections are $\bsu$-equivalences:
\[
 \kos {\bsx}{U^n}  \thra  U^n/\fm U^n\quad\text{and}\quad
  U^n/\fm U^n\thra (U^n/\fm U^n)_0\,.
\]
\end{lemma}

\begin{proof}
Consider the canonical morphisms
\[
f^n\colon\kos {\bsx}{U^n} \thra U^n/\bsx U^n \quad\text{and}\quad g^n\colon U^n/\bsx U^n\thra U^n/\fm U^n\,.
\]
We prove that the sequences $(f^n)$ and $(g^n)$ are $\bsu$-equivalences, which implies that the composition $(g^nf^n)$  is also a $\bsu$-equivalence, as claimed.

It is easy to see that the homology of the cone of $f^n$ satisfies
\[
\HH {i+1}{\cone(f^n)} \cong 
\begin{cases}
\HH {i}{\bsx; U^n} & \text{for $i\ge 1$.} \\
0 & \text{otherwise}
\end{cases}
\]
Thus $(f^n)$ is a $\bsu$-equivalence if and only if
\[
\lim_{n\to \infty}\frac{\rank_k\HH {i}{\bsx;U^n}_j}{u_n} = 0 \quad\text{for $i\ge 1$ and $j\in\bbZ$.}
\]
It remains to observe that the lim Ulrich property expressed in condition \eqref{eq:lim-Ulrich}, in \ref{ch:lim-Ulrich}, implies the limit above is $0$, that is to say, $(f^n)$ is a $\bsu$-equivalence, as desired.

For each $n$ one has that
\[
\rank_k (\fm U^n/\bsx U^n) =  \rank_k (U^n/\bsx U^n) - u_n 
\]
so condition (2) defining the lim Ulrich property of $(U^n)_{n\geqslant 1}$ yields
\[
\lim_{n\to \infty} \frac{\rank_k (\fm U^n/\bsx U^n)}{u_n} = 0\,.
\]
The cone of the morphism $g^n\colon U^n/\bsx U^n\to U^n/\fm U^n$ satisfies 
\[
\HH {i}{\cone(g^n)} \cong 
\begin{cases}
(\fm U^n)/(\bsx U^n) & \text{for $i= 1$} \\
0 & \text{otherwise.}
\end{cases}
\]
It follows that the sequence of maps $(g^n)$ is a $\bsu$-equivalence.

This completes the proof that the sequence $(g^nf^n)$ is a $\bsu$-equivalence.

Condition (3) in \ref{ch:lim-Ulrich} is equivalent to the condition that
\[
\lim_{n\to\infty} \frac{\rank_k (U^n/\fm U^n)_{i\ne 0}}{u_n} = 0\,.
\]
Since the homology of the mapping cone of the surjection 
\[
(U^n/\fm U^n) \to (U^n/\fm U^n)_0
\]
is $(U^n/\fm U^n)_{i\ne 0}$ in degree 1 and zero otherwise, it follows that the sequence above is a $\bsu$-equivalence, as desired.
\end{proof}

\begin{chunk}
\label{ch:limU-check}
It is clear from the proof of Lemma~\ref{le:limU-check} that when $(U^n)_{n\geqslant 0}$ is a lim Ulrich sequence, the sequences $(f^n)$ and $(g^n)$ are $\bsu$-equivalent in the stronger sense that 
\[
\lim_{n\to \infty}\frac{\rank_k \HH {i}{\cone(f^n)}}{\nu_R(U^n)} = 0 = \lim_{n\to \infty}\frac{\rank_k \HH {i}{\cone(g^n)}}{\nu_R(U^n)}\,.
\]
That is to say, the ranks of the homology modules, and not just their graded pieces, of the cones are asymptotically zero with respect to $(\nu_R(U^n))$. These conditions are equivalent to the lim Ulrich property of the sequence $(U^n)_{\geqslant 1}$. 

Moreover, it is not hard to verify that if the sequence $(g^nf^n)$ is a $\bsu$-equivalence (in either sense) then so are the sequences 
$(f^n)$ and $(g^n)$; the converse is clear.
\end{chunk}

\end{chunk}

\begin{lemma}
\label{le:perfect}
Let $R$ be a graded $k$-algebra with a linear system of parameters  $\bsx$, and $(U^n)_{n\geqslant 1}$  a lim Ulrich sequence in $\grmod R$. For any $P$ in $\dbbcat(\grmod R)$ and integers $i,j$ there is an equality
\[
\lim_{n \to \infty} \frac{\rank_k \HH i{P  \lotimes_R \kos{\bsx}{ U^n}}_j}{\nu_R(U_n)} = \rank_k \HH i{P \lotimes_R k}_j\,.
\]
\end{lemma}

\begin{proof}
Set $u_n\colonequals \nu_R(U^n)$ and $\bsu\colonequals (u_n)$.  Set 
\[
F^n\colonequals \kos{\bsx}{U^n} \quad\text{and}\quad  G^n\colonequals (U^n/\fm U^n)_0\,,
\]
and let $h^n\colon F^n\to G^n$ be the composition of maps 
\[
\kos{\bsx}{U^n}\to U^n/\fm U^n \to (U^n/\fm U^n)_0\,.
\]
A composition of $\bsu$-equivalences is a $\bsu$-equivalence, so it follows from  Lemma~\ref{le:limU-check} that the sequence $(h^n)$ is a $\bsu$-equivalence. Evidently for $i<0$ and all $n$ one has
\[
\HH i{\kos{\bsx}{U^n}}=0\qquad\text{and}\qquad \HH i{G^n}=0 \,.
\]
Thus Proposition~\ref{pr:uideal}(2) applies to yield that the sequence $(P\lotimes_R h^n)$ is also a $\bsu$-equivalence.  Since $R$ acts on $G^n$ via the surjection $R\to k$, one gets isomorphisms
\[
\HH i{P\lotimes_R G^n} \cong \HH i{(P\lotimes_Rk) \otimes_k G^n} \cong \HH i{P\lotimes_Rk} \otimes_k G^n\,.
\]
Since $G^n$ lives in internal degree $0$, this justifies the second equality below.
\begin{align*}
\lim_{n \to \infty} \frac{\rank_k \HH i{P\lotimes_R G^n}_j }{u_n} 
	&= \lim_{n \to \infty} \frac{\rank_k \HH i{P\lotimes_R G^n}_j }{\rank_k(G^n)} \\
	&=\rank_k \HH i{P\lotimes_R k}_j \,.
\end{align*}
The first one holds because the surjection $U^n/\fm U^n\to G^n$ is a $\bsu$-equivalence, by Lemma~\ref{le:limU-check},  so 
 Proposition~\ref{pr:uiso} applies.
\end{proof}

\begin{lemma}
\label{le:RtoX}
Let $k$ be a field and $R$ a graded $k$-algebra admitting a linear system of parameters.  Set $X\colonequals \proj R$ and  $\mcL\colonequals \widetilde{R(1)}$. The coherent sheaf $\mcL$ is invertible, ample, globally generated, and satisfies $\mcL^{\otimes t}=\widetilde{R(t)}$ for each $t$.
\end{lemma}

\begin{proof}
Let $\bsx\colonequals x_1,\dots,x_d$ be a linear system of parameters for $R$. To see that $\mcL$ is invertible and $\mcL^{\otimes t}=\widetilde{R(t)}$, one can use the same  argument as in the proof of \cite[Proposition 5.12 (a) (b)]{Hartshorne:1977}. The point  is that the affine open sets $\{D_+(x_i)\}_{i=1}^d$ cover $X$ since $\bsx$ is a linear system of parameters for $R$. The sheaf $\mcL$ is globally generated since the $\bsx$ are global sections of  $\mcL$ that have no common zero locus, again since $\bsx$ is a system of parameters. For $t\gg 0$, the $t$'th Veronese  of $R$ is standard graded, so $\mcL^{\otimes t}$ is very ample, and hence $\mcL$ is ample; see \cite[Theorem~II.7.6]{Hartshorne:1977}.
\end{proof}

\begin{chunk}
\label{ch:regularity}
Let $R$ be a graded $k$-algebra admitting a linear Noether normalization, with irrelevant maximal ideal $\fm$, and $M$ a finitely generated $R$-module. We write $\lch i{\fm}M$ for the $i$th local cohomology module of $M$ supported on $\fm$; see \cite[p.~143]{Bruns/Herzog:1998}. These are graded $R$-modules, with ${\lch i{\fm}M}_t=0$ for $t\gg 0$.  The  Castelnuovo-Mumford \emph{regularity} of $M$ is the integer
\[
\reg_R M \colonequals \max\{i+t\mid {\lch i{\fm}M}_{t}\ne 0 \}\,.
\]
If  $P\to R$ is a finite map of graded $k$-algebras, one has an isomorphism
\[
\lch i{P_{\geqslant 1}}M\cong \lch i{\fm}M
\]
so $\reg_RM=\reg_PM$.  Thus the regularity of $M$ can be computed with respect any linear Noether normalization for $R$. 

Set $X\colonequals \proj R$, and let $\mcF$ be a coherent sheaf on $X$ such that $\depth \mcF_x\ge 1$ at each closed point $x\in X$.
Set 
\[
M\colonequals \varGamma_*(\mcF) \colonequals \bigoplus_{t\in\bbZ} \mcF(t)\,.
\]
The condition on depth of $\mcF$ ensures that this is a finitely generated $R$-module; see Lemma~\ref{le:sheaf-depth}. One then has an equality
\begin{equation}
\label{eq:regularity}
\reg_R M = \max\{i+t \mid \CH {i-1}X{\mcF(t)}\ne 0 \text{ and } i\ge 2 \}\,.
\end{equation}
This holds because 
\begin{equation} 
\label{eq:sheaf-lch}
\begin{aligned}
&\lch 0{\fm}M=0=\lch 1{\fm}M, \quad\text{and} \\
&{\lch i{\fm}M}_{t}\cong \CH {i-1}X{\mcF(t)}  \quad \text{for  $i\ge 2$.}
\end{aligned}
\end{equation}

\end{chunk}

Next we apply the Theorem~\ref{th:limU-sheaves} to construct a lim Ulrich sequences in $\grmod R$. 

\begin{theorem}
\label{th:limU-graded-modules}
Let $k$ be an infinite perfect field of positive characteristic. Let $R$ be a graded $k$-algebra with $\dim R\ge 2$ and admitting a linear Noether normalization. There exists a lim Ulrich sequence $(U^n)_{n\geqslant 0}$  in $\grmod R$ satisfying the following conditions:
\begin{enumerate}[\quad\rm(1)]
\item
For each $n$, $\depth_R(U^n)\ge 2$ and  $\depth_{R_\fp}  (U^n_\fp) \ge \depth R_\fp$ for $\fp$ in $\proj R$. 
\item
There exist integers $t_0$ and $t_1$ such that for each $n$ one has
\[
U_j^n=0 \text{ for $j<t_0$ } \quad\text{and}\quad \reg U^n \le t_1\,.
\]
\end{enumerate}
\end{theorem}

\begin{proof}
Set $X\colonequals \proj R$ and $\mcL$ the coherent sheaf on $X$ defined by $R(1)$. Since $R$ admits a linear system of parameters, Lemma~\ref{le:RtoX} yields that ($X,\mcL$) is a Noether pair. Let $(\mcF_n)_{n\geqslant 0}$ be the lim Ulrich sequence of sheaves on $X$ given by Theorem~\ref{th:limU-sheaves}, and set
\[
U^n \colonequals \varGamma_{*}(\mcF_n) \quad\text{for $n\ge 1$.}
\]
This is an $R$-module, and it is finitely generated because $\depth( \mcF_n)_x\ge 1$ at all closed points  $x\in X$, by construction; see Lemma~\ref{le:sheaf-depth}.

(1) From \eqref{eq:sheaf-lch} one gets $\depth_R U^n\ge 2$. For any $x\in X$, one has an isomorphism $((\mcF_n)_x)\otimes_k k(t)\cong (U^n)_\fp$, where $\fp\in \proj R$ corresponds to $x$.  Thus  the claim about depths follows from the corresponding property for the sheaves $\mcF_n$.

(2) Since $U^n_t = \CH 0X{\mcF_n(t)}$, property (2) in \ref{ch:limUsheaves-P} implies $U^n_j=0$ for $j<t_0$. The claim about regularity is immediate from \eqref{eq:sheaf-lch}, given property (3) in \ref{ch:limUsheaves-P}.
\end{proof}

Given the preceding theorem, one can argue as in the proof of \cite[Theorem A]{Ma:2023}, see also \cite[Theorem B]{Ma:2023}, to obtain the following result.

\begin{corollary}
With $R$ as above,  if $(R_\fm,\fm R_\fm) \to S$ is a flat local map of noetherian local rings, then $e(R)\le e(S)$.\qed
\end{corollary}

\section{Boij-S\"oderberg theory}
In this section we recall  basics of Boij-S\"oderberg theory, due mainly to Boij, Eisenbud, Erman,  Schreyer, and S\"oderberg~\cites{Boij/Soderberg:2012, Boij/Soderberg:2008, Eisenbud/Erman:2017, Eisenbud/Schreyer:2010, Eisenbud/Schreyer:2009}.

\subsection*{Codimension sequences and degree sequences}
Throughout we fix a non-negative integer $d$, soon to be the dimension of a graded ring. By a \emph{codimension sequence} (for $d$) we mean a non-decreasing sequence $\bsc\colonequals (c_i)_{i\in \bbZ}$ where each $c_i$  is an element of the set $\{\varnothing,0,\dots,d,\infty\}$, with ordering
\[
\varnothing < 0 < 1 < \cdots < d< \infty\,.
\]
For any non-negative integer $c$ the constant sequence $(c)_{i\in\bbZ}$ is also denoted $c$.  We consider the collection of codimension sequences with the natural partial order: 
\[
(c_i)_{i\in\bbZ} \le (c'_i)_{i\in\bbZ}\qquad\text{if $c_i \le c'_i$ for each $i$.}
\]
A \emph{degree sequence} is an increasing sequence of integers $t_a < t_{a+1} <\cdots < t_{a+l}$, with $0\le l\le d$.  It is convenient to view it as an infinite sequence
\[
\bst\colonequals (\dots, -\infty, -\infty, t_a,\dots, t_{a+l}, \infty,\infty,\dots)
\]
with $t_i$ in position $i$. The integer $l$ is the \emph{codimension} of $\bst$, denoted $\codim(\bst)$; we define $\inf(\bst) \colonequals a$. The sequence $\bst$ is \emph{compatible} with a codimension sequence $\bsc$ if 
\[
0\le c_a \le \codim(\bst) \le c_{a+1} \qquad\text{for $a=\inf(\bst)$.}
\]
In what follows we  consider the $\bbQ$-vector space 
\begin{equation}
\label{eq:Vvec}
V\colonequals \bbQ^{(\bbZ\times \bbZ)} =  \bigoplus_{i,j\in\bbZ}\bbQ\,.
\end{equation}
Each degree sequence $\bst$ determines an element $\beta(\bst)$ of $V$ given by
\begin{equation}
\label{eq:HK}
\beta(\bst)_{i,j}\colonequals  
\begin{cases}
\frac{ \prod_{n\ne a} | t_n - t_a|}{ \prod_{n\ne i} |t_n - t_i|} & \text{for $a\le i\le a+l$ and $j=t_i$,} \\
0 &\text{for all other values of $i,j$.}
\end{cases}
\end{equation}
Here $a\colonequals \inf(\bst)$ and $l=\codim(\bst)$. Given a codimension sequence $\bsc$, we set
\begin{equation}
\label{eq:codim-cone}
\bcone^{\bsc}_{d}\colonequals \left\{
\begin{aligned}
&\text{cone in $V$ spanned by $\beta(\bst)$ as $\bst$ ranges}\\
&\text{over degree sequences compatible with $\bsc$}
\end{aligned}
\right\}\,.
\end{equation}
This is the smallest subset of $V$ containing the $\beta(\bst)$, for the permitted $\bst$, and closed under addition and multiplication by non-negative rational numbers.

\subsection*{The finite topology on $V$}
Given a finite subset $Y \subset \bbZ\times \bbZ$ we write $V_Y$ for the subspace of $V$ spanned by the corresponding coordinate vectors. Thus a vector $v\in V$ is in $V_Y$ if and only if $v_{i,j}=0$ for $(i,j)\notin Y$.  We topologize $V$ by giving it the finite topology: a subset $U$ of $V$ is open if and only if its intersection with any finite dimensional subspace of $V$ is open in the Euclidean topology. This is equivalent to the condition that $U \cap V_Y$ is open in $V_Y$, equipped with its Euclidean topology, for all finite subsets $Y$ of $\bbZ\times \bbZ$. 

Observe that the cone spanned by vectors $v^1,\dots,v^n$ in the positive orthant of $V$ is contained in $V_Y$ if and only if each $v^i$ is in $V_Y$. This has the following consequence.

\begin{lemma}
\label{le:local-finiteness}
With $Y$ as above, for any codimension sequence $\bsc$ the cone $\bcone^{\bsc}_{d}\cap V_Y$ is spanned by vectors $\beta(\bst)$ where $\bst$ is any degree sequence compatible with $\bsc$ and satisfying the condition
\[
(n,t_n) \in Y \qquad\text{for $\inf(\bst) \le n \le \inf(\bst)+\codim(\bst)$.}
\]
Thus $\bcone^{\bsc}_{d}\cap V_Y$ is the cone spanned by a finite collection of vectors, and hence the subset $\bcone^{\bsc}_{d}\subset V$ is  closed in the finite topology. \qed
\end{lemma}

The result above means that if a sequence of elements $(\beta_n)_{n\geqslant 0}$ in $\bcone^{\bsc}_{d}$ converges in the finite topology to an element $\beta$ of $V$, then $\beta$ must also belong to $\bcone^{\bsc}_{d}$.

\subsection*{Graded modules over graded rings}
As in Section~\ref{se:limU}, let $k$ be a field and $R=\{R_i\}_{i\geqslant 0}$ a finitely generated graded $k$-algebra with $R_0=k$. Set $\fm\colonequals R_{i\geqslant 1}$.

\begin{chunk}
Let $M$ be an $R$-complex in $\dbbcat(\grmod R)$. For any pair of integers $(i,j)$, the \emph{Betti number} of $M$ in degree $(i,j)$ is 
\[
\beta^R_{i,j}(M) \colonequals \rank_k \Tor^R_i(k,M)_j\,.
\]
The \emph{Betti table} of $M$ is the array  $\beta^R(M)\colonequals (\beta^R_{i,j}(M))_{i,j}$. 
\end{chunk}

\subsection*{Finite free complexes}
With $R$ as above,  we write  $\perf(R)$ for the class of finite free complexes of graded $R$-modules, that is to say, complexes of the form
\[
0\lra F_b \lra \cdots \lra F_a \lra 0
\]
where each $F_i$ is a finitely generated graded free $R$-module, and the differential is homogenous of degree zero.  
Each finite free complex is quasi-isomorphic to a \emph{minimal} one: a finite free complex $F$  whose differential $d$ satisfies $d(F)\subseteq \fm F$. When $F$ is minimal, one has 
\[
F_i \cong \bigoplus_j R(-j)^{\beta^R_{i,j}(F)}\,.
\]
Since $F$ is finite free, $\beta^R_{i,j}(F)$ is nonzero only for finitely many pairs $(i,j)$, so $\beta^R(F)$, the Betti table of $F$, 
is an element in the $\bbQ$-vector space $V$ from \eqref{eq:Vvec}. 

\subsection*{Codimension}
Let $\codim_RM$ be the \emph{codimension} of a finitely generated graded $R$-module $M$, namely, the height of its annihilator ideal. Thus $\codim_RM=\infty$ if and only if $M=0$.  Let $\bsc$ be a codimension sequence, as defined above.  Set
\begin{equation}
\label{eq:perfR}
\perf^{\bsc}(R)\colonequals \left\{F\in \perf(R)\left| 
	\begin{aligned}
	\text{ $\beta^R_{i,j}(F)=0$ when $c_i=\varnothing$, and}\\
	\text{$\codim \HH iF\ge c_i$ when $c_i>\varnothing$}
	\end{aligned}\right.\right\}\,.
\end{equation}
For instance, for the constant sequence $\dim R$ the objects in $\perf^{\dim R}(R)$ are precisely the finite free complexes with homology of finite length. Moreover, $\perf^{0}(R)=\perf(R)$. Clearly, if $\bsc'$ is another codimension sequence with $\bsc\le \bsc'$, then $\perf^{\bsc}(R)\supseteq \perf^{\bsc'}(R)$.

Consider the cone in the $\bbQ$-vector space $V$ from \eqref{eq:Vvec}, spanned by the Betti tables of the finite free complexes $F$ in $\perf^{\bsc}(R)$:
\begin{equation}
\label{eq:betti-cone}
\bcone^{\bsc}(R)\colonequals \sum_{F\in \perf^{\bsc}(R)} \bbQ_{\geqslant 0} \beta^R(F)\,.
\end{equation}

\subsection*{Pure complexes}
A finite free complex $F$ is \emph{pure} if there exists a degree sequence $\bst$ as  above such that for $a\colonequals \inf(\bst)$ and $l\colonequals \codim \bst$, the following conditions hold:
\begin{enumerate}[\quad\rm(1)]
\item
$\beta_{i,j}(F)\ne 0$ if and only if $j=\bst_{i}$, that is to say,  the complex $F$ is quasi-isomorphic to minimal complex of the form
\[
0\lra R(-t_{a+l})^{b_{a+l}} \lra \cdots \lra  R(-t_{a})^{b_{a}} \lra 0\,;
\]
\item
The $R$-module $\HH aF$ is Cohen-Macaulay with $\codim \HH aF\ge l$;
\item
$\HH nF =0$ for $n \ne a$.
\end{enumerate}
In particular, $\shift^{-a}F$ is the free resolution of a finitely generated $R$-module, namely, $\HH aF$.  It follows from the result of Herzog and K\"uhl~\cite{Herzog/Kuhl:1984} that the Betti table of such an $F$ has the form
\[
\beta^R(F) = b \cdot \beta(\bst)\,,
\]
where $b\colonequals \beta_{a}^R(F)$ and $\beta(\bst)$ is as in \eqref{eq:HK}.

\subsection*{Boij-Soderberg theory}
Let $S\colonequals k[x_1,\dots,x_d]$ be a standard graded polynomial ring; thus the degree of each $x_i$ is $1$.  The following result is due to Eisenbud and Erman, and builds on work of Eisenbud and Schreyer, and Boij and S\"oderberg. The cone $\bcone^{\bsc}_{d}$ is defined in  \eqref{eq:codim-cone}.

\begin{theorem}{\cite[Theorem~3.1]{Eisenbud/Erman:2017}}
\label{th:EE}
Each degree sequence has an associated pure object in $\perf(S)$.  For any codimension sequence $\bsc$ for $S$ one has $\bcone^{\bsc}(S)=\bcone^{\bsc}_{d}$.\qed
\end{theorem}

By a codimension sequence $\bsc$ \emph{for} a graded ring $R$ we mean a codimension sequence for the integer $\dim R$. Combining the preceding theorem with Lemma~\ref{le:local-finiteness} gives the following result.

\begin{corollary}
\label{co:bst}
For a standard graded polynomial ring $S$, and any codimension sequence $\bsc$ for $S$, the subset $\bcone^{\bsc}(S)\subset V$ is closed in  the finite topology. \qed
\end{corollary}

\section{Cones of Betti tables}

\begin{chunk}
\label{ch:notation}
Throughout this section $k$ is an infinite perfect field (possibly of characteristic zero) and $R$ a finitely generated graded $k$-algebra admitting a linear Noether normalization, say $A$.
\end{chunk}

\begin{theorem}
\label{th:main-bs}
Let $R$ be as in \ref{ch:notation}. If $\bsc$ is a codimension sequence for $R$, then  $\bcone^{\bsc}(R)\subseteq \bcone^{\bsc}_{d}$. Equality holds if $c_i\ne\infty$ for all $i$, or $R$ is Cohen-Macaulay.
\end{theorem}

The proof of the theorem when $k$ has positive characteristic is given in \ref{ch:main-bs-p}. The characteristic zero case is handled by reduction to positive characteristic; see \ref{ch:main-bs-0}. We  record  a couple of expected consequences. 

When $R$ is a standard graded polynomial ring, the bounds given below on the multiplicities (whose definition is recalled in \ref{ch:linear-sop}) were conjectured by Huneke and Srinivasan, and proved by Eisenbud and Schreyer~\cite{Eisenbud/Schreyer:2009} as a consequence of their proof of the conjectures of Boij and S\"oderberg~\cite{Boij/Soderberg:2008}. 

\begin{corollary}
\label{co:main-bs}
Let $R$ be as in \ref{ch:notation}, and let  $M$ be a nonzero finitely generated, graded, $R$-module generated in degree zero, and of finite projective dimension. Then 
\[
 e(M) \le \frac{e(R)\beta_0(M)}{c!} \prod_{i=1}^c\max\{j\mid \beta^R_{i,j}(M)\ne 0\}\,,
\]
for $c\colonequals \codim M$. When $M$ is perfect, one has also the lower bound:
\[
  e(M) \ge \frac{e(R)\beta_0(M)}{c!} \prod_{i= 1}^c\min\{j\mid \beta^R_{i,j}(M)\ne 0\}\,.
\]
Each of the inequalities above is strict, unless $M$ has a pure resolution.
\end{corollary}

\begin{proof}
We deduce these bounds from Theorem~\ref{th:main-bs} following the argument in \cite{Boij/Soderberg:2008}. 
 
There is nothing to check if $\dim R=0$ so we can assume $d\colonequals \dim R\ge 1$. Let $A$ be a Noether normalization of $R$ as in \ref{ch:notation}. 
It follows from Lemma~\ref{le:multiplicity} that the multiplicity of a graded $R$-module coincides with its multiplicity as an $A$-module.

Fix an integer $c$ with $0 \leq c \leq d$. We  look at Betti tables of $R$-modules of finite projective dimension and codimension at least $c$. So let $\bsc \colonequals (\dots, \varnothing, c, \infty, \dots)$ with $c$ in position $0$ and $\bcone^{\bsc}_{d}$ the corresponding rational cone; see \eqref{eq:codim-cone}. It is thus the positive rational cone spanned by $\beta(\bst)$ where $\bst$ is any degree sequence of the form
\[
(\dots-\infty, t_0,\dots, t_l, \infty,\dots)
\]
with $c \leq l \leq \dim R$. We know that  $\bcone^{\bsc}(R)\subseteq \bcone^{\bsc}_{d}$, by Theorem~\ref{th:main-bs}.

Let $\hilb_R(t)$ be the Hilbert series of $R$, and for $\beta \in \bcone^{\bsc}_{d}$, set
\[
\hilb_\beta(t) \colonequals \hilb_R(t) g_\beta(t),
\quad
\text{where}
\quad
g_\beta(t) \colonequals \sum_{i,j} (-1)^i \beta_{i,j} t^j \in \bbZ[t, t^{-1}]\,.
\]
Since $R$ admits a linear system of parameters, $\hilb_R(t) = f_R(t)/(1-t)^{d}$ for some polynomial $f_R(t)$ and $e(R) = f_R(1)$. If $M\in\grmod R$ has finite projective dimension, then $\hilb_\beta(t) = \hilb_M(t)$, where $\beta=\beta^R(M)$; see \cite[Lemma~4.1.13]{Bruns/Herzog:1998}.

For $\beta \colonequals \beta(\bst)$ with $\bst$ a degree sequence as above we have that $g_\beta(t)$ is divisible by $(1-t)^{l}$, with $l=\codim(\bst)$, and no higher power. In particular, we have that $g_\beta(t)$ is divisible by $(1-t)^c$ for all $\beta \in \bcone^{\bsc}_{d}$. It thus makes sense to define
\[
e(\beta) \colonequals e(R) \left(\frac{g_\beta(t)}{(1-t)^c}\right)_{t=1} \quad\text{for $\beta \in \bcone^{\bsc}_{d}$.}
\]
The function $e(-)$ depends on the chosen $c$ and $e(\beta(\bst)) = 0$ if $\codim(\bst) \ge c+1$.

Recall that one has 
\[
\codim_RM = \dim R - \dim_R M=d-\dim_R M
\]
for any $M\in\grmod R$ of finite projective dimension; see \cite[Th\'eor\`eme 2]{Peskine/Szpiro:1974}.

\begin{claim} 
If $M\in\grmod R$ has finite projective dimension, and codimension at least $c$, then $e(\beta^R(M)) = e_{d-c}(M)$; see \ref{ch:linear-sop}.
\end{claim}

Indeed, since $M$ has dimension at most $d-c$,  one has that
\[
\hilb_M(t) = \frac{f_M(t)}{(1-t)^{d-c}}
\]
for some polynomial $f_M(t)$ and then  $e_{d-c}(M) = f_M(1)$. On the other hand, for $\beta=\beta^R(M)$,  the discussion above yields
\begin{align*}
\hilb_M(t) 
	&= \hilb_\beta(t) \\
	& = \hilb_R(t) g_\beta(t) \\
	& = \frac{f_R(t)}{(1-t)^{d}} g_\beta(t)\\
	& = \frac{(f_R(t) g_\beta(t)/(1-t)^c}{(1-t)^{d-c}}\,.
\end{align*}
Thus $f_M(t) = f_R(t) g_\beta(t)/(1-t)^c$; evaluating this equality at $t = 1$ yields the claim.

Once we know that the multiplicity of a module can be extracted from its Hilbert series, the rest of the argument is as in \cite[Theorem 4.6]{Boij/Soderberg:2008}. 
\end{proof}

We write $\bcone^{\mathrm{short}}(R)$ for the positive cone in $V$, defined in \eqref{eq:Vvec}, spanned by $\beta^R(F)$ as $F$ ranges over the short complexes over $R$; see \ref{ch:short-complexes}.

\begin{corollary}
With $R$ and $A$ as in \ref{ch:notation}, there is an equality $\bcone^{\mathrm{short}}(R) = \bcone^{\mathrm{short}}(A)$.
\end{corollary}

\begin{proof}
Let $\bsc$ be the codimension sequence $(\dots,\varnothing,\dim R,\dim R,\dots)$ with the first occurrence of $\dim R$ is in degree $0$.
Then $\bcone^{\bsc}(R)= \bcone^{\bsc}(A)$, by Theorem~\ref{th:main-bs}. Intersecting these cones with the cone $V^{[0,d]}$ yields the stated equality.
\end{proof}

In proving Theorem~\ref{th:main-bs} it will be expedient to use the notion of depth for complexes, and some results from \cite{Iyengar:1999}. In what follows we state the necessary definitions and results for local rings, with the implicit understanding that the corresponding statements in the graded context also hold.

\begin{chunk}
\label{ch:depth}
Let $R$ be a local ring with residue field $k$. By the \emph{depth} of an $R$-complex $M$  of  $R$-modules we mean the least integer $i$, possibly $\pm \infty$, such that $\mathrm{Ext}_R^i(k,M)\ne 0$; see \cite{Foxby/Iyengar:2001} for alternative descriptions.  Given a bounded complex $M$ let $\sup\HH *M$ denote the supremum of integers $i$ such that $\HH iM\ne 0$. For such an $R$-complex $M$ there is an inequality
\begin{equation}
\label{eq:depth-estimate}
\depth_RM \ge - \sup\HH *M\,.
\end{equation}
Equality holds if and only if $\depth_R \HH sM=0$ for $s\colonequals \sup \HH *M$. These claims are contained in \cite[Proposition~2.7(2) and Theorem~2.3]{Iyengar:1999}.

One version of the Auslander-Buchsbaum formula for complexes reads: If $F$ is finite free $R$-complex with $\HH{} F\ne 0$, then for any $R$-complex $M$ one has
\begin{equation}
\label{eq:AB-formula}
\begin{aligned}
\depth_R(F\otimes_RM) 
	&= \depth_RM - \pdim_RF \\
	& = \depth_RM - \depth R + \depth_RF\,.
\end{aligned}
\end{equation}
The first equality is \cite[Corollary~2.2]{Iyengar:1999}, and the second follows from the first applied with $M=R$.
\end{chunk}

\begin{lemma}
\label{le:AB-lemma}
Fix a codimension sequence  $\bsc$ for $R$,  an $R$-complex $F$ in $\perf^{\bsc}(R)$ and a finitely generated $R$-module $U$ satisfying $\depth_{R_\fp}U_{\fp}\ge \depth R_\fp$ for all $\fp$ in $\spec R\setminus \{\fm\}$. For any integer $i$, if
$c_i<\infty$, or $\depth_RU \ge \depth R$, then
\[
\codim_R \HH i{F\otimes_RU}\ge c_i\,.
\]
When $R$ is a finitely generated graded $k$-algebra, and $F$ and $U$ are also graded, then it suffices that the condition on depth holds for $\fp$ in $\proj R$.
\end{lemma}

\begin{proof}
The key observation is the following:

\begin{claim}
Assume $\depth_RU\ge \depth R$ as well. Then $\sup \HH *{F\otimes_RU} \le \sup \HH *F$.
\end{claim}

Indeed, since $\sup\HH *F$ does not increase on localization, we can further localize at prime minimal in the support of $\HH s{F\otimes_RU}$, for $s\colonequals \sup \HH*{F\otimes_RU}$, and assume it has nonzero finite length. Then the claim about equality in \eqref{eq:depth-estimate} yields the first equality below:
\begin{align*}
-  \sup \HH *{F\otimes_RU}	
	&=  \depth_R (F\otimes_RU) \\
	& = \depth_RU - \depth R + \depth_R F \\
	&\ge \depth_RF\\
	& \ge - \sup \HH *F\,.
\end{align*}
The second equality is~\eqref{eq:AB-formula}.  The first  inequality holds by our hypotheses on $U$, and another application of \eqref{eq:depth-estimate} yields the last one. This settles the claim.

Fix an $i$. The case $c_i=\infty$ and $\depth_RU\ge \depth R$ follows from the claim.

Suppose $c_i<\infty$ and choose $\fq\in\proj R$ such that  $\height\fq < c_i$.  Since $c_i\le\dim R$ it follows that $\fq\ne \fm$.  The choice of $\fq$ means that $\sup \HH *{F_\fq} < i$, because $F$ is in $\perf^{\bsc}(R)$, and we have to verify that 
\[
\sup \HH *{(F\otimes_RU)_\fq} < i\,.
\]
This is covered by the claim, applied to the local ring $R_\fq$, the $R_\fq$-complex $F_\fq$, and $U_\fq$. The applies since $\fq$ is a non-maximal prime.

In the graded context, it suffices to observe that the conditions on depth for $\fp$ in $\proj R$ imply the condition for all primes in $\spec R\setminus\{\fm\}$ by \cite[Theorem~1.5.9]{Bruns/Herzog:1998}.
\end{proof}

\begin{chunk}
Let $\bsc$ be a codimension sequence of the form
\[
\bsc \colonequals (\dots,\varnothing, \varnothing, c_a,\dots,c_{b},\infty,\infty,\dots)
\]
with $0\le c_a\le c_b\le \dim R$. Consider the sequence 
\[
\bsc' \colonequals (\dots,\varnothing, \varnothing, c_a,\dots,c_{b},d,d,\dots)
\]
where $d=\dim R$. Evidently, $\bsc'\le \bsc$. 

\begin{lemma}
\label{le:trick} 
With $\bsc$ and $\bsc'$ as above, if $\bcone^{\bsc'}(R) \subseteq \bcone^{\bsc'}_{d}$, then $\bcone^{\bsc}(R) \subseteq \bcone^{\bsc}_{d}$.
\end{lemma}

\begin{proof}
Pick $F \in \perf^{\bsc}(R)$. Then one has 
\[
\pdim_RF = \depth R - \depth_RF \le d + b\,.
\]
Consequently, one has
\[
 \beta_{i, j}^R(F) = 0\quad\text{for $i\ge d + b + 1$ and all $j$.}
 \]
 This fact will be used further below.

Since $\bsc' \leq \bsc$, we have  $\bcone^{\bsc}(R) \subseteq \bcone^{\bsc'}(R)$, and also for $A$. Thus, by hypothesis,  $\beta^R(F)$ is in $\bcone^{\bsc'}(A)$. Theorem~\ref{th:EE} gives a decomposition
\[
\beta^R(F) = \sum_nr_n \beta^A(G^n) 
\]
 as a finite sum, with $r_n>0$ and  $G^n \in \perf^{\bsc'}(A)$, with $G^n$ the shifted resolution of an $A$-module. That is, $\HH i{G^n}= 0$ for all but one value, say $i_n$, of $i$.  Moreover, by the definition of $\bsc'$, we have
 \[
 \codim_A \HH{i_n}{G^n} =
 \begin{cases}
 &c_{i_n} \text{ or more if $a\le i_n\le b$}\\
 &\dim A \text{ if $i_n>b$.}
 \end{cases}
 \]
In particular, when $i_n > b$ the Auslander-Buchsbaum formula yields 
\[
\beta_{d+b+1, j}(G^n) \ne 0 \quad \text{for some $j$.}
\]
But then $\beta^R_{d +b+1, j}(F) \ne 0$ as well, a contradiction. Thus $a\le i_n\le b$ for each $n$, that is to say,  $G^n$ is in $\perf^{\bsc}(A)$.
\end{proof}
\end{chunk}

The last part of Theorem~\ref{th:main-bs} holds over any field $k$, so we record it separately.

\begin{proposition}
\label{pr:main-bs}
Let $R$ be as in \ref{ch:notation} and  $\bsc$ a codimension sequence for $R$.
If $c_i\ne\infty$ for all $i$, or $R$ is Cohen-Macaulay, then  $\bcone^{\bsc}(R)\supseteq \bcone^{\bsc}_{d}$.
\end{proposition}

\begin{proof}
With $A$ as in \ref{ch:notation}, one has $\bcone^{\bsc}(A)=\bcone^{\bsc}_{d}$, by Theorem~\ref{th:EE}, so the desired inclusion is that $\bcone^{\bsc}(R)\supseteq \bcone^{\bsc}(A)$.   When $c_i<\infty$, this is verified in the course of the proof of \cite[Theorem~9.1]{Eisenbud/Erman:2017}; see in particular, \cite[Lemma~9.5, 9.6]{Eisenbud/Erman:2017} and \cite{Miller/Speyer:2008}. 
We give the argument, for one key equality in the proof was not justified in \emph{op.~cit.}

One can assume $c_i=\varnothing$ for $i<0$ and $c_0$ is finite.  Since $\bcone^{\bsc}(A)$ is spanned by (shifts) of the Betti tables of finitely generated graded $A$-modules, it suffices to prove that if $M$ is a finitely generated with $\codim_AM\ge c_0$, then $\beta^A(M)$ is in $\bcone^{\bsc}(R)$. By \cite{Miller/Speyer:2008}, pulling back $M$ along a $k$-algebra automorphism of $A$, one can suppose that the $R$-modules $\Tor^A_i(R,M)$ have finite length for each $i\ge 1$. Let $F$ be the minimal free resolution of $M$ over $A$, so the finite free complex of $R$-modules $R\otimes_AF$ satisfies 
\[
\beta^R(R\otimes_AF) = \beta^A(F) = \beta^A(M)\,.
\]
Since $\HH i{R\otimes_AF}\cong \Tor^A_i(R,M)$ has finite length for $i\ge 1$, and $c_i<\infty$, one gets
\[
\codim_R\HH i{R\otimes_AF} \ge \dim R \ge c_i \quad\text{for $i\ge 1$.}
\]
It remains to observe that
\[
\codim_R \HH 0{R\otimes_AF} = \codim_R(R\otimes_AM) = \codim_A(M)\,.
\]
The equality on the right is by Lemma~\ref{le:codim}; this is glossed over in the proof of \cite[Theorem~9.1]{Eisenbud/Erman:2017}. This completes the proof in the case when each $c_i$ is finite. 

Suppose that $R$ is Cohen-Macaulay. Fix an $A$-complex $G$ in $\bcone^{\bsc}(A)$ and consider the $R$-complex $F\colonequals R\otimes_AG$.
Evidently $\beta^R(F) = \beta^A(G)$, so for the desired inclusion one has only to verify that $F$ is in $\bcone^{\bsc}(R)$. Since $R$ is Cohen-Macaulay and finite as an $A$-module for any finitely generated $R$-module $M$ one has
\[
\codim_RM =\dim R - \dim_RM = \dim A - \dim_AM= \codim_AM\,.
\]
Thus to verify that $F$ is in $\bcone^{\bsc}(R)$ it suffices to verify that 
\[
\codim_A \HH i{R\otimes_AG} \ge c_i \quad\text{for each $i$.}
\]
This follows from Lemma~\ref{le:AB-lemma}, applied to the $A$-complex $G$ and for $U=R$.
\end{proof}

\begin{chunk}
\label{ch:regularity-betti-table} 
With $A$ as in \ref{ch:notation}, the regularity of a finitely generated graded $R$-module $M$, see \ref{ch:regularity}, can be read from the Betti table of $M$ over $A$, as follows: 
\[
\reg_RM  = \max\{j \mid \beta^A_{i,i+j}(M) \ne 0 \}\,.
\]
See \cite[Theorem~4.3.1]{Bruns/Herzog:1998}.  
\end{chunk}

\begin{chunk}[Proof of Theorem~\ref{th:main-bs} when the characteristic of $k$ is nonzero]
\label{ch:main-bs-p}
Let $A$ be as in \ref{ch:notation}. Given Theorem~\ref{th:EE}, the desired inclusion is  $\bcone^{\bsc}(R)\subseteq \bcone^{\bsc}(A)$.  By Lemma~\ref{le:trick} we can assume the codimension sequence $\bsc$ satisfies $c_i\le \dim R$ for each $i$.

Let $(U^n)_{n\geqslant 0}$ be a lim Ulrich sequence of graded $R$-modules provided by Theorem~\ref{th:limU-graded-modules}. This is where we need the hypothesis that the characteristic of $k$ is positive. Fix an  $R$-complex $F$ in $\perf^{\bsc}(R)$ and consider the $R$-complexes $F\otimes_R U^n$. 

For any $M $ in $\grmod R$ one has $\codim_A M\ge \codim_RM$, so Lemma~\ref{le:AB-lemma} yields that, viewed as an $A$-complex by restriction of scalars (the finite free resolution of) the complex $F\otimes_R U^n$ is in $\bcone^{\bsc}(A)$.  Set $u_n \colonequals \nu_R (U^n)$. Since $A$ is regular, the canonical surjection $\kos{\bsx}A\xrightarrow{\sim} k$ is a free resolution over $A$, so for all integers $i,j$ one has equalities
\begin{align*}
\beta^A_{i,j}(F \otimes_R U^n)  
	&= \rank_k \HH i{F \otimes_R U^n \otimes_A \kos{\bsx}A}_j \\
	&= \rank_k \HH i{F \otimes_R U^n \otimes_R \kos{\bsx}R}_j \\
	&=\rank_k \HH i{F \otimes_R \kos{\bsx}{U^n}}_j
\end{align*}
These equalities and Lemma~\ref{le:perfect} yield
\[
\lim_{n \to \infty} \frac{\beta_{i,j}^A(F\otimes_R U^n)}{u_n} = \beta_{i,j}^R(F)\,.
\]
This proves that $\beta^R(F)$ is in the closure, in the point-wise topology on $V$, of the sub-cone of $\bcone^{\bsc}(A)$ spanned by all the $\beta^A(F\otimes_RU^n)$, for $n\ge 0$. We claim that that the same is true in the finite topology on $V$, and to that end it suffices to verify:

\begin{claim}
For all but finitely many pairs $(i,j)$ one has $\beta^A_{i,j}(F\otimes_RU^n)=0$ for all $n$. 
\end{claim}

Let $t_0$ and $t_1$ be as in Theorem~\ref{th:limU-graded-modules}. Given that $\reg_R(U^n)$ can be computed in terms of $\beta^A_{i,j}(U^n)$---see  \ref{ch:regularity-betti-table}---we get that
\[
\beta^A_{i,j}(U^n) = 0\quad \text{unless $i\in [0,\dim A]$ and  $j\in [t_0,t_1+\dim A]$.}
\]
The stated claim follows from the usual change of rings spectral sequence:
\[
\mathrm{E}_{p,q}^2 \colonequals \Tor^R_p(F,k)\otimes_k \Tor^A_q(U^n,k) \Longrightarrow \Tor^A_{p+q}(F\otimes_RU^n,k)\,.
\]

Since $\bcone^{\bsc}(A)$ is closed in the finite topology on $V$, by Corollary~\ref{co:bst}, we deduce that $\bcone^{\bsc}(R)\subseteq \bcone^{\bsc}(A)$, as desired.

The last part of the statement follows from Proposition~\ref{pr:main-bs}.
\end{chunk}

\subsection*{Characteristic zero}
Next we explain how to deduce the characteristic zero case of Theorem~\ref{th:main-bs} from that of positive characteristic. The key is the following result, whose proof is a standard argument.

\begin{lemma}
\label{le:reduction-char-p}
Let $k$ be a field of characteristic zero,  $R$ a finitely generated graded $k$-algebra, and  fix $F$ in $\perf(R)$. There exists a field $k_p$ of positive characteristic $p$, a finitely generated graded $k_p$-algebra $R_p$, and an $R_p$-complex $F_p$ in $\perf(R_p)$ such that \begin{gather*}
\beta^{R_p}(F_p) = \beta^R(F)\quad\text{and}\\
\codim_{R_p} \HH i{F_p} = \codim_{R} \HH i{F}\quad\text{for each $i$.} 
\end{gather*}
Moreover, if the $k$-algebra $R$ admits a linear normalization, then the $k_p$-algebra $R_p$ can be chosen to admit one as well.
\end{lemma}

\begin{proof}
Write $R\colonequals k[x_1,\dots,x_n]/J$, where $J=(f_1,\dots,f_t)$ is a homogenous ideal. Pick a finitely generated $\bbZ$-subalgebra $W$ of $k$ such that the coefficients of the $f_i$ and the coefficients in the entries in the differentials of $F$ are in $W$.
Thus we have a  graded ring $R_W\colonequals W[x_1,\dots,x_n]/(f_1,\dots, f_t)$ over $W$ viewed as a ring concentrated in degree $0$, and such that $R= k\otimes_W R_W$.

When $R$ admits a linear Noether normalization, then by \cite[(2,1,8)(h)]{Hochster/Huneke:2100}, one can invert elements in $W$ to ensure that the inclusion $W[x_1,\dots,x_d]\subseteq R_W$ is module-finite, with cokernel $W$-free. It follows that for all maximal ideal $\fm$ of $W$ base-change along $W\to W/\fm$ induces a map
\[
(W/\fm)[x_1,...,x_d]\to R_{W/\fm}\,,
\]
that is a linear Noether normalization of $R_{W/\fm}$.

Write $\beta^R(F) = (\beta_{i,j})$. Evidently there exists a diagram (not necessarily a complex) of finite free $R_W$-modules
\[
F_W\colonequals \cdots \lra \bigoplus R_W(-j)^{\beta_{i+1,j}} \lra  \bigoplus R_W(-j)^{\beta_{i,j}} \lra  \bigoplus R_W(-j)^{\beta_{i-1,j}}\lra\cdots
\]
such that $F=k\otimes_W F_W$. With $Q$ the field of fractions of $W$ one has an isomorphism
\[
k\otimes_W F_W\cong k\otimes_Q(Q\otimes_WF_W)
\]
In particular, $Q\otimes_WF_W$ is a complex of $R_Q$-modules and the codimension of the $R_Q$-module $\HH i{Q\otimes_WF_W}$ equals that of the $R$-module $\HH iF$, for each $i$. 

Inverting further elements in $W$, one can ensure that $F_W$ is complex and for each maximal ideal $\fm$ of $W$ and integer $i$ one has
\[
\codim_{R_{W/\fm}} \HH i{F_{W/\fm}}=  \codim_{R_W} \HH i{F_{W}} = \codim_R \HH iF\,.
\]
Indeed, this follows from \cite[(2.1.14)(g) last sentence, and (2.3.9)(c)]{Hochster/Huneke:2100}; the first reference tells us we can  localize at elements in $W$ to preserve the annihilators of the homology modules (note that $k\otimes_W\HH i{F_W}= \HH iF$), and the second one says that by 
inverting more elements in $W$ the height of an ideal is preserved; see  (2.3.3) and (2.3.4) in \emph{op.~cit.} for notation.

Finally, pick a maximal ideal $\fm$ of $W$ such that $\HH {}{F_W}\otimes_{R_W}W/\fm W$ is nonzero and set $F_p\colonequals W/\fm \otimes_W F_W$. It is clear the $R_p$-complex $F_p$ has the same Betti numbers as $F$; the codimensions of their homology modules is also the same, by construction. It remains to note that $W/\fm$ has positive characteristic $p$.
\end{proof}

\begin{chunk}[Proof of Theorem~\ref{th:main-bs} when the characteristic of $k$ is zero]
\label{ch:main-bs-0}
This follows from  Lemma~\ref{le:reduction-char-p} and the case of positive characteristic, settled in \ref{ch:main-bs-p}. 
\end{chunk}

We do not know  whether Theorem~\ref{th:bs-cohomology} also holds over fields of characteristic zero.  The  method of reduction to positive characteristics exploited above does not appear to work in the context of cohomology tables.

\appendix

\section{Codimension}
In this section $A\hookrightarrow R$ is an integral extension of noetherian rings. For any finitely generated $A$-module $M$, it is easy to see that there is an inequality
\[
\codim_R(R\otimes_AM) \le \codim_AM\,.
\]
Equality holds when, for instance, the going-down theorem holds for the extension $A\to R$, but not always; see Example~\ref{ex:codim-failure}. Here is one other situation in which equality holds; this comes up in the proof of Proposition~\ref{pr:main-bs}.

\begin{lemma}
\label{le:codim}
Let $A\to R$ as above and $M$ a finitely generated $A$-module with $\pdim_AM$ finite and such that
\[
\length_R \Tor^A_i(R,M)  < \infty \quad\text{for all $i\ge 1$.}
\]
There is an equality $\codim_R(R\otimes_AM) = \codim_AM$.
\end{lemma}

\begin{proof}
One has to only to verify that $\codim_R(R\otimes_AM)\ge \codim_AM$, that is to say  
\[
\height \fq \ge \codim_AM \quad\text{for all $\fq$ in $\supp_R(R\otimes_AM)$.}
\]

Let $f\colon \spec R \to \spec A$ be the map induced by $A\to R$. Then
\[
\supp_R(R\otimes_AM) = f^{-1}(\supp_AM)\,.
\]
Fix $\fq$ minimal in $f^{-1}(\supp_AM)$. If $\fq$ is the maximal ideal of $R$, then 
\[
\height \fq = \dim R=\dim A\ge \codim_AM\,,
\]
as desired. We can thus assume that $\fq$ is not maximal in $R$. 

Set $\fp\colonequals \fq\cap A$; this is in $\supp_AM$ and so contains a minimal element, say $\fp'$, in $\supp_AM$. It suffices to verify
\[
\height \fq \ge  \height \fp'\,.
\]
Since $\fq$ is not maximal in $\spec R$, the hypotheses yield
\[
\Tor^A_i(R,M)_\fq \cong \Tor^{A_\fp}_i(R_\fq,M_\fp) = 0 \quad\text{for $i\ge 1$.}
\]
Thus the $R_\fq$-module $(R\otimes_AM)_\fq$ has finite projective dimension; it is also of nonzero finite length, because $\fq$ is minimal in $\supp_R(R\otimes_AM)$. This gives the first and the second equalities below:
\begin{align*}
\dim R_\fq 
	&= \depth R_\fq \\
	&= \pdim_{R_\fq}(R\otimes_AM)_\fq \\
	&=\pdim_{R_\fq}(R_\fq\otimes_{A_\fp} M_\fp)\\
	&= \pdim_{A_\fp}M_\fp \\
	&\ge \pdim_{A_{\fp'}}M_{\fp'}\\
	&= \dim A_{\fp'}\,.
\end{align*}
The fourth one holds as $A_\fp \to R_\fq$ is local and, as $A_\fp$-modules, $R_\fq$ and $M_\fp$ are Tor-independent. The inequality is clear. The last equality holds hold because $M_{\fp'}$ has finite length, since $\fp'$ is minimal in $\supp_AM$. Thus $\height \fq \ge \height \fp'$.
\end{proof}

\begin{example}
\label{ex:codim-failure}
Let $R\colonequals k[x,y,z]/(x,y)\cap (z)$ and $A$ the subring $k[x+z,y]$. Then for $M\colonequals A/yA$ one has 
\[
\codim_AM = 1 \quad\text{whereas}\quad \codim_A(R\otimes_AM) = \codim_R(R/yR)=0\,.
\]
In geometric language, $\spec(R)$ is the union of a line and a plane, both passing through the origin, $\spec(A)$ is a plane, and $\spec(R) \to \spec(A)$ any linear map that is finite (any generic such map would do).  Then for one ``bad" line in $\spec(A)$, its inverse image in $\spec(R)$ is a union of the exceptional line in $\spec(R)$ and another line passing through the origin. This gives an example where the codimension in $A$ is 1 but the codimension in $R$ drops to $0$.
\end{example}

\end{document}